\documentclass[11pt]{article}

\usepackage{lipsum}

\newcommand\blfootnote[1]{%
  \begingroup
  \renewcommand\thefootnote{}\footnote{#1}%
  \addtocounter{footnote}{-1}%
  \endgroup
}

\usepackage[utf8]{inputenc} %
\usepackage[T1]{fontenc}    %
\usepackage{hyperref}       %
\usepackage{url}            %
\usepackage{booktabs}       %
\usepackage{nicefrac}       %
\usepackage{microtype}      %

\usepackage{amsthm,amsfonts,amsmath,amssymb,epsfig,color,float,graphicx,verbatim, enumitem}

\newtheorem{theorem}{Theorem}[section]
\newtheorem{lemma}[theorem]{Lemma}

\newtheorem{proposition}[theorem]{Proposition}

\newtheorem{claim}[theorem]{Claim}

\newtheorem{definition}[theorem]{Definition}

\newtheorem{fact}[theorem]{Fact}

\newcommand{\eps}{\epsilon}

\newcommand{\E}{{\mathbb E}}
\newcommand{\R}{{\mathbb R}}
\newcommand{\I}{{\mathbb I}}
\newcommand{\bP}{{\mathbb P}}
\newcommand{\cM}{{\mathcal M}}
\newcommand{\cS}{{\mathcal S}}
\newcommand{\cT}{{\mathcal T}}
\newcommand{\Var}{\mathbb{V}}

\newcommand{\trace}{\operatorname{tr}}
\newcommand{\cov}{\operatorname{Cov}}
\newcommand{\srank}{\operatorname{r}}

\newcommand{\littlesum}{\mathop{\textstyle \sum}}

\def\colorful{0}

\ifnum\colorful=1
\newcommand{\new}[1]{{\color{red} #1}}

\else
\newcommand{\new}[1]{{#1}}

\fi

\title{Outlier Robust Mean Estimation with Subgaussian Rates 
\\via Stability\blfootnote{Authors are in alphabetical order.}}

\author{
Ilias Diakonikolas\thanks{Supported by NSF Award CCF-1652862 (CAREER) and a Sloan Research Fellowship.}\\
University of Wisconsin-Madison\\
{\tt ilias@cs.wisc.edu}\\
\and
Daniel M. Kane\thanks{Supported by NSF Award CCF-1553288 (CAREER) and a Sloan Research Fellowship.}\\
University of California, San Diego\\
{\tt dakane@cs.ucsd.edu}\\
\and
Ankit Pensia\thanks{Supported by NSF Award CCF-1740707 (TRIPODS).}\\ University of Wisconsin-Madison\\
{\tt ankitp@cs.wisc.edu}
}

\begin{document}

\maketitle

\begin{abstract}
We study the problem of outlier robust high-dimensional mean estimation under a finite 
covariance assumption, and more broadly under finite low-degree moment assumptions. We consider a standard stability condition from the recent robust statistics literature and prove that, except with exponentially small failure probability, there exists a large fraction of the inliers satisfying this condition. As a corollary, it follows that a number of recently developed algorithms for robust mean estimation, including iterative filtering and non-convex 
gradient descent, give optimal error estimators with (near-)subgaussian rates. 
Previous analyses of these algorithms gave significantly suboptimal rates. 
As a corollary of our approach, we obtain the first computationally efficient
algorithm with subgaussian rate for outlier-robust mean estimation 
in the strong contamination model under a finite covariance assumption.
\end{abstract}

\thispagestyle{empty}

\setcounter{page}{0}

\newpage

\section{Introduction} \label{sec:introduction}

\subsection{Background and Motivation} \label{ssec:background}
Consider the following problem:
For a given family $\mathcal{F}$ of distributions
on $\mathbb{R}^d$, estimate the mean of an unknown  
$D \in \mathcal{F}$, given access to i.i.d. samples from $D$.
This is the problem of (multivariate) mean estimation
and is arguably {\em the} most fundamental statistical task.
In the most basic setting where $\mathcal{F}$ is the family of 
high-dimensional Gaussians, the empirical mean is well-known 
to be an optimal estimator --- in the sense that it achieves
the best possible accuracy-confidence tradeoff and is easy to compute. 
Unfortunately, the empirical mean is known to be highly suboptimal if we relax the aforementioned modeling assumptions.  
In this work, we study high-dimensional mean estimation 
{\em in the high confidence regime} when the underlying family 
$\mathcal{F}$ is only assumed to satisfy bounded moment conditions (e.g., finite covariance). 
Moreover, we relax the ``i.i.d. assumption'' and aim to obtain estimators that are robust 
to a constant fraction of adversarial outliers. 

Throughout this paper, we focus on the following data contamination model (see, e.g.,~\cite{DKKLMS16}) 
that generalizes several existing models, including Huber's contamination model~\cite{Huber64}. 

\begin{definition}[Strong Contamination Model] \label{def:adv}
Given a parameter $0< \epsilon < 1/2$ and a distribution family $\mathcal{F}$ on $\mathbb{R}^d$, the \emph{adversary} operates as follows: The algorithm specifies the number of samples $n$, and $n$ samples are drawn from some unknown $D \in \mathcal{F}$. The adversary is allowed to inspect the samples, remove up to $\epsilon n$ of them and replace them with arbitrary points. This modified set of $n$ points is then given as input to the algorithm. We say that a set of samples is {\em $\epsilon$-corrupted} if it is generated by the above process.
\end{definition}

The parameter $\epsilon$ in Definition~\ref{def:adv} is the fraction of outliers 
and quantifies the power of the adversary. Intuitively, among our input samples, 
an unknown $(1-\epsilon)$ fraction are generated from a
distribution of interest and are called {\em inliers}, and the rest are called {\em outliers}. 

We note that the strong contamination model is strictly stronger than Huber's contamination
model. Recall that in Huber's contamination model~\cite{Huber64}, the adversary generates 
samples from a mixture distribution $P$ of the form $P = (1-\epsilon) D + \epsilon N$, 
where $D \in \mathcal{F}$ is the unknown target distribution and $N$ is an adversarially chosen noise distribution. 
That is, in Huber's model the adversary is oblivious to the inliers and is only allowed to add outliers.

In the context of robust mean estimation, we want to design an algorithm (estimator) 
with the following performance: 
Given any $\epsilon$-corrupted set of $n$ samples from an unknown distribution $D \in \mathcal{F}$, 
the algorithm outputs an estimate $\widehat{\mu} \in \mathbb{R}^d$ of the target mean $\mu$ of $D$ 
such that {\em with high probability} the $\ell_2$-norm $\|\widehat{\mu} - \mu \|$ is small. 
The ultimate goal is to obtain a {\em computationally efficient estimator with optimal confidence-accuracy tradeoff}.
For concreteness, in the proceeding discussion we focus on the case that $\mathcal{F}$ is the family 
of all distributions on $\mathbb{R}^d$ with bounded covariance, 
i.e., any $D \in \mathcal{F}$ has covariance matrix $\Sigma \preceq  I$.
(We note that the results of this paper apply for the more general setting where $\Sigma \preceq \sigma^2 I$,
where $\sigma>0$ is unknown to the algorithm.)

Perhaps surprisingly, even for the special case of $\epsilon = 0$ (i.e., without adversarial contamination), 
designing an optimal mean estimator in the high-confidence regime is far from trivial. 
In particular, it is well-known (and easy to see) that the empirical mean achieves highly sub-optimal rate. 
A sequence of works in mathematical statistics (see, e.g.,~\cite{Catoni12, Minsker15, Devroye2016, LM19-aos}) 
designed novel estimators with improved rates, culminating
in an optimal estimator~\cite{LM19-aos}. See~\cite{LugosiM19-survey} for 
a survey on the topic. The estimator of \cite{LM19-aos} 
is based on the median-of-means framework and 
achieves a ``subgaussian'' performance guarantee:
\begin{equation}\label{eqn:mom-optimal}
\|\widehat{\mu} - \mu\|  = O(\sqrt{d/n} + \sqrt{\log(1/\tau)/n}) \;,
\end{equation}
where $\tau>0$ is the failure probability. \new{The error rate \eqref{eqn:mom-optimal} is 
information-theoretically optimal for any estimator and matches
the error rate achieved by the empirical mean on Gaussian data.}
Unfortunately, the estimator of~\cite{LM19-aos} is not efficiently computable. 
In particular, known algorithms to compute it have running time exponential in the dimension $d$. 
Related works~\cite{Minsker15, PBR19} provide computationally efficient estimators
alas with suboptimal rates. The first polynomial time algorithm achieving the optimal rate
~\eqref{eqn:mom-optimal} was given in~\cite{Hop18}, using a convex program 
derived from the Sums-of-Squares method. Efficient algorithms with improved
asymptotic runtimes were subsequently given in~\cite{CFB19, DepLec19,LLVZ19}.

We now turn to the outlier-robust setting ($\epsilon>0$) for the constant confidence regime, 
i.e., when the failure probability $\tau$ is a small universal constant. The statistical foundations 
of outlier-robust estimation were laid out in early work by the robust statistics community, 
starting with the pioneering works of \cite{Tukey60} and \cite{Huber64}. For example, 
the minimax optimal estimator satisfies:
\begin{equation}\label{eqn:robust-constant-prob}
\|\widehat{\mu} - \mu\|  = O(\sqrt{\epsilon}+\sqrt{d/n}) \;.
\end{equation}
Until fairly recently however, all known polynomial-time estimators attained
sub-optimal rates. Specifically, even in the limit when $n \rightarrow \infty$, 
known polynomial time estimators achieved error of $O(\sqrt{\epsilon d})$, i.e., scaling
polynomially with the dimension $d$. 
Recent work in computer science, starting with~\cite{DKKLMS16, LaiRV16},
gave the first efficiently computable outlier-robust estimators 
for high-dimensional mean estimation. 
For bounded covariance distributions,~\cite{DKK+17, SteinhardtCV18}
gave efficient algorithms with the right error guarantee of $O(\sqrt{\epsilon})$.
Specifically, the filtering algorithm of~\cite{DKK+17} is known to achieve a near-optimal rate
of $O(\sqrt{\epsilon}+\sqrt{d \log d /n})$.
 
In this paper, we aim to achieve the best of both worlds. In particular, we ask the following question:
\begin{center}
{\em Can we design {\em computationally efficient} estimators with subgaussian rates\\ 
and optimal dependence on the contamination parameter $\epsilon$?}
\end{center} 
Recent work~\cite{LugosiM19robust} gave an {\em exponential time}
estimator with optimal rate in this setting. Specifically, ~\cite{LugosiM19robust} showed that
a multivariate extension of the trimmed-mean achieved the optimal error of
\begin{equation}\label{eqn:trimmed-optimal}
\|\widehat{\mu} - \mu\|  = O( \sqrt{\epsilon} + \sqrt{d/n} + \sqrt{\log(1/\tau)/n}) \;.
\end{equation}
We note that \cite{LugosiM19robust} posed as an open question the existence of a computationally
efficient estimator achieving the optimal rate~\eqref{eqn:trimmed-optimal}. 
Two recent works~\cite{DepLec19,LLVZ19} gave efficient estimators with subgaussian
rates that are outlier-robust in the {\em additive} contamination model --- a {\em weaker} model
than that of Definition~\ref{def:adv}. \new{Prior to this work, no  polynomial time algorithm 
with optimal (or near-optimal) rate was known in the strong contamination model of Definition~\ref{def:adv}.}
As a corollary of our approach, we answer the question of~\cite{LugosiM19robust} in the affirmative 
(see Proposition~\ref{prop:subgaussian}). In the following subsection, we describe our results
in detail.

\subsection{Our Contributions} \label{ssec:results}

At a high-level, the main conceptual contribution of this work is in showing
that several previously developed computationally efficient algorithms 
for high-dimensional robust mean estimation achieve near-subgaussian rates
or subgaussian rates (after a simple pre-processing). A number of these algorithms
are known to succeed under a standard {\em stability} condition (Definition~\ref{def:stability}) 
-- \new{a simple deterministic condition on the empirical mean and covariance of a finite point set.}
We will call such algorithms {\em stability-based.}

Our contributions are as follows:
\begin{itemize}[leftmargin=*]
\item We show (Theorem~\ref{ThmStabilityBddCov}) that given a set of i.i.d. samples from 
a finite covariance distribution, except with exponentially small failure probability, 
there exists a large fraction of the samples satisfying the stability condition. 
As a corollary, it follows (Proposition~\ref{CorBddCovDirect}) that {\em any} stability-based robust 
mean estimation algorithm achieves optimal error with (near-)subgaussian rates. 

\item We show an analogous probabilistic result (Theorem~\ref{ThmStabilityHighMom}) for 
known covariance distributions (or, more generally, spherical covariance distributions) 
with bounded $k$-th moment, for some $k \geq 4$. 
As a corollary, we obtain that {\em any} stability-based robust mean estimator 
achieves optimal error with (near-)subgaussian rates (Proposition~\ref{prop:higher-moments-stab}.)

\item For the case of finite covariance distributions, we show (Proposition~\ref{prop:subgaussian}) 
that a simple pre-processing step followed by any stability-based robust mean estimation 
algorithm yields optimal error and subgaussian rates.
\end{itemize}
To formally state our results, we require some terminology and background.

\paragraph{Basic Notation}
For a vector $v \in \R^d$, we use $\|v\|$ to denote its $\ell_2$-norm.
For a square matrix $M$, we use $\trace(M)$ to denotes its trace, and $\|M\|$ to denote its spectral norm. 
We say a symmetric matrix $A$ is PSD (positive semidefinite) if $x^TAx \geq 0$ for all vectors $x$.
For a PSD matrix $M$, we use $\srank(M)$ to denote its stable rank (or intrinsic dimension), i.e., $\srank(M):= \trace(M)/ \|M\| $.
For two symmetric matrices $A$ and $B$, we use $\langle A, B\rangle $ to denote the trace inner product $\text{tr}(AB)$ and  say $A \preceq B$ when $B -A $ is PSD.

We use $[n]$ to denote the set $\{1,\dots, n\}$ and 
$\mathcal{S}^{d-1}$ to denote the $d$-dimensional unit sphere. 
We use $\Delta_n$ to denote the probability simplex on $[n]$, 
i.e., $\Delta_n =  \{w \in \R^n: w_i \geq 0 , \littlesum_{i=1}^n w_i = 1  \}$.
For a multiset $S = \{x_1,\dots,x_n\} \subset \R^d$ of cardinality $n$  and $w \in \Delta_n$, we use $\mu_w$ 
to denote its weighted mean $\mu_w = \sum_{i=1}^nw_ix_i$. Similarly, we use $\overline{\Sigma}_w$ to denote its 
weighted second moment matrix (centered with respect to $\mu$)  $\overline{\Sigma}_w = \littlesum_{i=1}^n w_i(x_i - \mu)(x_i - \mu)^T$.
For a set $S \subset \R^d$, we denote $\mu_S = (1/|S|)\sum_{x \in S} x$ and 
$\overline{\Sigma}_S = (1/|S|) \littlesum_{x \in S} (x - \mu)(x - \mu)^T $ to denote 
the mean and (central) second moment matrix with respect to the uniform distribution on $S$.

For a set $E$, we use $\I( x \in E )$ to denote the indicator function for event $E$. For simplicity, we use $\I(x \geq t)$ to denote the indicator function for the event $E= \{x : x\geq t\}$.
For a random variable $Z$, we use $\Var(Z)$ to denote its variance.
We use $d_{\text{TV}}(p,q)$ to denote the total variation distance between distributions $p$ and $q$.

\paragraph{Stability Condition and Robust Mean Estimation.}

We can now define the stability condition:

\begin{definition}[see, e.g.,~\cite{DK20-survey}]\label{def:stability}
Fix $0 < \epsilon < 1/2$ and $ \delta \geq \epsilon$. 
A finite set $S \subset \R^d$ is $(\epsilon,\delta)$-stable with respect to mean $\mu \in \R^d$ and $\sigma^2$ 
if for every $S' \subseteq S$ with $|S'| \geq (1 - \epsilon) |S|$, 
the following conditions hold: (i) $\| \mu_{S'} - \mu\| \leq \sigma \delta$, and 
(ii) $\|\overline{\Sigma}_{S'}  - \sigma^2 I\| \leq \sigma^2\delta^2/\epsilon$.
\end{definition}

The aforementioned condition or a variant thereof is used
in every known outlier-robust mean estimation algorithm.
Definition~\ref{def:stability} requires that after restricting to a $(1-\epsilon)$-density subset $S'$,
the sample mean of $S'$ is within $\sigma\delta$ of the mean $\mu$,
and the sample variance of $S'$ is $ \sigma^2(1\pm \delta^2/\epsilon)$ in every direction.
(We note that Definition~\ref{def:stability} is intended for distributions with covariance $\Sigma \preceq \sigma^2I$). 
We will omit the parameters $\mu$ and $\sigma^2$ when they are clear from context. 
In particular, our proofs will focus on the case $\sigma^2 = 1$, which can be achieved by scaling the datapoints appropriately.

A number of known algorithmic techniques previously used for robust mean estimation, 
including convex programming based methods~\cite{DKKLMS16, SteinhardtCV18, ChengDG18},
iterative filtering~\cite{DKKLMS16, DKK+17, DHL19}, and even first-order methods~\cite{CDGS20,ZhuJS2020-gradient},
are known to succeed under the stability condition. Specifically, prior work has established
the following theorem:

\begin{theorem}[Robust Mean Estimation Under Stability, see, e.g.,~\cite{DK20-survey}] \label{ThmStabDK19}  
Let $T \subset \R^d$ be an $\epsilon$-corrupted version of a set $S$ with the following properties:
$S$ contains a subset $S' \subseteq S$ such that $|S'| \geq (1 - \epsilon)|S|$ and 
$S'$ is $(C \epsilon, \delta)$ stable with respect to $\mu \in \R^d$ and $\sigma^2$ , for a sufficiently large constant $C>0$.
Then there is a polynomial-time algorithm, that on input $\epsilon, T$, 
computes $\widehat{\mu}$ such that $\|\widehat{\mu} - \mu\| = O(\sigma\delta)$. 
\end{theorem}

We note in particular that the iterative filtering algorithm~\cite{DKK+17, DK20-survey} 
(see also Section~2.4.3 of \cite{DK20-survey}) is a very simple and practical stability-based algorithm. 
While previous works made the assumption that the upper bound parameter $\sigma^2$ is known to the algorithm, 
we point out in Appendix~\ref{AppUnkCov} that essentially the same algorithm and analysis works for unknown 
$\sigma^2$ as well.

\paragraph{Our Results.}
Our first main result establishes the stability of a subset of i.i.d. points drawn from a 
distribution with bounded covariance. 

\begin{theorem} \label{ThmStabilityBddCov}
Fix any $0< \tau<1$.
Let $S$ be a multiset of $n$ i.i.d. samples from a distribution on $\R^d$
with mean $\mu$ and covariance $\Sigma$ .
Let $ \epsilon' =  \Theta(\log(1/\tau)/n + \epsilon) \leq c$, for a sufficiently small constant $c>0$.
Then, with probability at least $1 - \tau$, there exists a subset $ S' \subseteq S$ such that 
$ |S'| \geq (1 - \epsilon')n$ and $S'$ is $ (2\epsilon', \delta)$-stable with respect to $\mu$ and $\|\Sigma\|$, where
$\delta = O(\sqrt{ (\srank(\Sigma) \log \srank(\Sigma)) / n} +  \sqrt{\epsilon} + \sqrt{\log(1/\tau)/n})$.
\end{theorem}

Theorem~\ref{ThmStabilityBddCov} significantly improves the probabilistic guarantees in prior work on robust mean estimation. This includes the \textit{resilience} condition of~\cite{SteinhardtCV18,ZHS19} and the 
\textit{goodness} condition of~\cite{DHL19}.

As a corollary, it follows that any stability-based algorithm for robust mean estimation 
achieves near-subgaussian rates.

\begin{proposition}\label{CorBddCovDirect}
Let $T$ be an $\epsilon$-corrupted set of $n$ samples from a distribution in $\R^d$ 
with mean $\mu$ and covariance $\Sigma$.
Let $\epsilon' =  \Theta(\log(1/\tau)/n + \epsilon) \leq c$ be given, for a constant $c > 0$.
Then any stability-based algorithm on input $T$ and $\epsilon'$, efficiently computes $\widehat{\mu}$ such that with probability at least 
$1 - \tau$, we have 
$ \|\widehat{\mu} - \mu\| = O(\sqrt{ (\trace{(\Sigma)} \log \srank(\Sigma)) / n} +  \sqrt{\|\Sigma\|\epsilon} + \sqrt{\|\Sigma\|\log(1/\tau)/n})$.
\end{proposition}

We note that the above error rate is minimax optimal in both $\epsilon$ and $\tau$, and the restriction of $\log(1/\tau)/n = O(1)$ is information-theoretically required~\cite{Devroye2016}.
In particular, the term $\sqrt{\log(1/\tau)/n}$ is {\em additive} as opposed to multiplicative.
The first term is near-optimal, up to the $\sqrt{\log \srank(\Sigma) }$ factor, which is at most $\sqrt{\log d}$ (recall that $\srank(\Sigma)$ denotes the stable rank of $\Sigma$, i.e., $\srank(\Sigma) = \trace(\Sigma)/ \|\Sigma\|$). 
Prior to this work, the existence of a polynomial-time algorithm 
achieving the above near-subgaussian rate in the strong contamination model was open. 
Proposition~\ref{CorBddCovDirect} shows that any stability-based algorithm suffices
for this purpose, and in particular it implies that the iterative filtering algorithm~\cite{DK20-survey} 
achieves this rate {\em as is}.

Given the above, a natural question is whether stability-based algorithms achieve 
subgaussian rates {\em exactly}, i.e., whether they match the optimal bound~\eqref{eqn:trimmed-optimal}
attained by the computationally inefficient estimator of~\cite{LugosiM19robust}. 
While the answer to this question remains open, we 
show that after a simple pre-processing of the data, stability-based estimators are indeed subgaussian.

The pre-processing step follows the median-of-means principle~\cite{NY83, JVV86, AlonMS99}. 
Given a multiset of $n$ points $x_1, \ldots, x_n$ in $\R^d$ and $k \in [n]$, we proceed as follows: 
\begin{enumerate}[leftmargin=*]
\item First randomly bucket the data into $k$ disjoint buckets of equal size (if $k$ does not divide $n$, remove some samples) and compute their empirical means $z_1,\ldots,z_k$.
\item Output an (appropriately defined) multivariate median of $z_1,\dots,z_k$.  
 \end{enumerate}
Notably, for the case of $\epsilon= 0$, all known  efficient mean estimators with subgaussian rates
use the median-of-means framework~\cite{Hop18,DepLec19,CFB19,LLVZ19}.

To obtain the desired computationally efficient robust mean estimators with subgaussian rates,
we proceed as follows:
\begin{enumerate}[leftmargin=*]
\item Given a multiset $S$ of $n$ $\epsilon$-corrupted samples, randomly group the data into $k = \lfloor\epsilon ' n \rfloor$
disjoint buckets, where $\epsilon' =  \Theta(\log(1/\tau)/n + \epsilon)$, 
and let $z_1,\dots,z_k$ be the corresponding empirical means of the buckets.

\item Run any stability-based robust mean estimator on input $\{ z_1,\dots,z_k \}$.
\end{enumerate}

Specifically, we show:

\begin{proposition}(informal) \label{prop:subgaussian}
Consider the same setting as in Proposition~\ref{CorBddCovDirect}.
Let $k = \lfloor \epsilon' n \rfloor $ and $z_1,\dots,z_k$ be the points after median-of-means pre-processing 
on the corrupted set $T$.
Then any stability-based algorithm, on input $\{ z_1,\dots,z_k \}$, computes $\widehat{\mu}$ 
such that with probability at least $1 - \tau$, it holds
 $ \|\widehat{\mu} - \mu\| = O(\sqrt{ \trace(\Sigma)  / n} +  \sqrt{\|\Sigma\|\epsilon} + \sqrt{\|\Sigma\|\log(1/\tau)/n})$.
\end{proposition}

Proposition~\ref{prop:subgaussian} yields the first computationally efficient algorithm
with subgaussian rates in the strong contamination model, answering the open question of~\cite{LugosiM19robust}.

To prove Proposition~\ref{prop:subgaussian}, we establish a connection between the median-of-means principle 
and stability. \new{In particular, we show that the key probabilistic lemma 
from the median-of-means literature~\cite{LM19-aos,DepLec19} also implies stability.}

\begin{theorem}(informal)\label{ThmBddCovMom}
Consider the setting of Theorem~\ref{ThmStabilityBddCov} and set $k = \lfloor\epsilon ' n \rfloor$. 
The set $\{ z_1,\ldots, z_k\}$, with probability $1 - \tau$, contains a subset of size at least $0.99k$ which is $(0.1, \delta)$-stable with respect to $\mu$ and $k \|\Sigma\|/n$, where 
$\delta = O(\sqrt{\srank(\Sigma) / k} +  1)$.
\end{theorem}

A drawback of the median-of-means framework is that the error dependence 
on $\epsilon$ does not improve if we impose stronger assumptions on the distribution. 
Even if the underlying distribution is an identity covariance Gaussian, 
the error rate would scale as $O(\sqrt{\epsilon})$, whereas the stability-based 
algorithms achieve error of $O(\epsilon \sqrt{\log(1/\epsilon)})$~\cite{DKKLMS16}.
Our next result establishes tighter error bounds for distributions with 
identity covariance and bounded central moments.

We say that a distribution has a bounded $k$-th central moment $\sigma_k$, if for all unit vectors $v$,
it holds $(\mathbb E (v^T(X - \mu))^k)^{1/k} \leq \sigma_k (\mathbb E (v^T(X - \mu))^2)^{1/2}$.
For such distributions, we establish the following stronger stability condition.

\begin{theorem} \label{ThmStabilityHighMom}
Let $S$ be a multiset of $n$ i.i.d. samples from a distribution on 
$\R^d$ with mean $\mu$, covariance $\Sigma = I$, and bounded central moment $\sigma_k$, for some $k\geq 4$.
Let $ \epsilon' =  \Theta(\log(1/\tau)/n + \epsilon) \leq c$, for a sufficiently small constant $c>0$.
Then, with probability at least $1 - \tau$, there exists a subset $ S' \subseteq S$ such that 
$ |S'| \geq (1 - \epsilon')n$ and $|S'|$ is $ (2\epsilon', \delta)$-stable with respect to $\mu$ and $\sigma^2=1$, where
$\delta = O(\sqrt{d \log d / n} + \sigma_k \epsilon^{1 - \frac{1}{k}}  + \sigma_4\sqrt{\log(1/\tau)/n})$.
\end{theorem}

As a corollary, we obtain the following result for robust mean estimation with high probability in the 
strong contamination model:

\begin{proposition}\label{prop:higher-moments-stab}
Let $T$ be an $\epsilon$-corrupted set of $n$ points from a distribution on $\R^d$ 
with mean $\mu$, covariance $\sigma^2I$, and $k$-th bounded central moment $\sigma_k$, for some $k\geq 4$.
Let $ \epsilon' =  \Theta(\log(1/\tau)/n + \epsilon) \leq c$ be given, for some $c > 0$.
Then any stability-based algorithm, on input $T$ and $\epsilon'$, efficiently computes $\widehat{\mu}$ 
such that with probability at least $1 - \tau$, we have
$\|\widehat{\mu} - \mu\| = O( \sigma(\sqrt{d \log d / n} + \sigma_k \epsilon^{1 - \frac{1}{k}}  + \sigma_4\sqrt{\log(1/\tau)/n}))$.
\end{proposition}

\new{We note that the above error rate is near-optimal up to the $\log d$ factor and the dependence
on $\sigma_4$. Prior to this work, no polynomial-time estimator achieving this rate was known.
Finally, recent computational hardness results~\cite{HL19} suggest that the assumption 
on the covariance above is inherent to obtain
computationally efficient estimators with error rate better than $\Omega(\sqrt{\epsilon})$,
even in the constant confidence regime.}

\subsection{Related Work}
Since the initials works~\cite{DKKLMS16, LaiRV16}, 
there has been an explosion of research activity on algorithmic aspects of 
outlier-robust high dimensional estimation by several communities. 
See, e.g.,~\cite{DK20-survey} for a recent survey on the topic.
In the context of outlier-robust mean estimation, a number of works
~\cite{DKK+17, SteinhardtCV18, ChengDG18, DHL19} have obtained efficient algorithms 
under various assumptions on the distribution of the inliers. 
Notably, efficient high-dimensional outlier-robust mean estimators
have been used as primitives for robustly solving machine learning tasks
that can be expressed as stochastic optimization 
problems~\cite{PrasadSBR2018, DiakonikolasKKLSS2018sever}.
The above works typically focus on the constant probability error regime and 
do not establish subgaussian rates for their estimators.

Two recent works~\cite{DepLec19,LLVZ19} studied the problem of outlier-robust mean estimation in the additive contamination model (when the adversary is only allowed to add outliers) and gave computationally efficient algorithms with subgaussian rates. Specifically, \cite{DepLec19} gave an SDP-based algorithm, which is very similar to the algorithm of~\cite{ChengDG18}. The algorithm of~\cite{LLVZ19} is a fairly sophisticated iterative spectral algorithm, building on~\cite{CFB19}. In the strong contamination model, non-constructive outlier-robust 
estimators with subgaussian rates were established very recently. Specifically,~\cite{LugosiM19robust} gave an exponential time estimator achieving the optimal rate. Our Proposition~\ref{prop:subgaussian} implies 
that a very simple and practical algorithm -- pre-processing followed by iterative 
filtering~\cite{DKK+17, DK20-survey} -- achieves this guarantee.

In an independent and concurrent work, Hopkins, Li, and Zhang \cite{HopLiZhang20} also studied the relation between median-of-means and stability for the case of bounded covariance.

\subsection{Organization}
In Section~\ref{sec:bounded_covariance}, we prove Theorem~\ref{ThmStabilityBddCov} that establishes the stability of points sampled from a finite covariance distribution. 
In Section~\ref{sec:bdd_cov_med_of_mean}, we establish the connection between median-of-means principle and stability to prove Theorem~\ref{ThmBddCovMom}.
Finally, Section~\ref{sec:bounded_central_moments} contains our 
results for distributions with identity covariance and finite central moments.

\section{Robust Mean Estimation for Finite Covariance Distributions} %
\label{sec:bounded_covariance}

\paragraph{Problem Setting}
Consider a distribution  $P$ in $\R^d$ with unknown mean $\mu$ and unknown covariance $\Sigma$.
We first note that it suffices to consider the distributions such that $\|\Sigma\| = 1$.
 Note that for covariance matrices $\Sigma$ with $\|\Sigma\| = 1$, we have $\srank(\Sigma) = \trace(\Sigma)$.
 In the remainder of this section, we will thus establish the $(\epsilon, \delta)$ stability with respect to $\mu$ and $\sigma^2 =1$, where $\delta = O(\sqrt{\trace(\Sigma)\log (\srank(\Sigma))/n} + \sqrt{\epsilon} + \sqrt{\log(1 / \tau)/n})$.

Let $S$ be a multiset of $n$ i.i.d. samples from $P$.
For the ease of exposition, we will assume that the support of $P$ is bounded, i.e., for each $i$, $\|x_i - \mu\| = O(\sqrt{ \trace(\Sigma)/ \epsilon})$ almost surely.
As we show in Section~\ref{SecProofThmBddCov}, we can simply consider the points violating this condition as outliers.

We first relax the conditions for stability in  the Definition~\ref{def:stability} in the following Claim~\ref{ClaimCovSuffStabMain}, proved in Appendix~\ref{AppCovSuff}, at an additional cost of $O(\sqrt{\epsilon})$.
 \begin{claim} (Stability for bounded covariance)
Let $R \subset \R^d$ be a finite multiset such that $\|\mu_R - \mu\| \leq \delta$, and 
$\|\overline{\Sigma}_R - I\| \leq \delta^2 / \epsilon$ for some $0 \leq \epsilon \leq \delta$. 
Then $R$ is $(\Theta(\epsilon), \delta')$ stable with respect to $\mu$ (and $\sigma^2 = 1$), where $\delta' =  O(\delta + \sqrt{\epsilon})$.
\label{ClaimCovSuffStabMain}
\end{claim}
Given Claim~\ref{ClaimCovSuffStabMain}, our goal in proving Theorem~\ref{ThmStabilityBddCov} is to show that with probability $1 - \tau$, there exists a set $S'\subseteq S$ such that $|S'| \geq (1 - \epsilon')n$,  $\|\mu_{S'} - \mu\| \leq \delta$ and $\|\overline{\Sigma}_S - I\| \leq \delta^2	/\epsilon'$, for some value of $\delta = O(\sqrt{ \trace(\Sigma) \log \srank(\Sigma)/n} + \sqrt{\epsilon} + \sqrt{\log(1/ \tau)/n} )$ and  $\epsilon' = \Theta(\epsilon + \log(1 / \tau)/n)$.

We first remark that the original set $S$ of $n$ i.i.d. data points does not satisfy either of the conditions in Claim~\ref{ClaimCovSuffStabMain}.
It does not satisfy the first condition because the sample mean is highly sub-optimal for heavy-tailed data~\cite{Catoni12}.
For the second condition, we note that the known concentration results for $\overline{\Sigma}_S$ are not sufficient.  
For example, consider the case of $\Sigma = I$ in the parameter regime of $\epsilon, \tau,$ and $n$ such that $\epsilon = O(\log(1/\tau)/n)  $ and $n = \Omega(d \log d / \epsilon)$ so that $\delta = O(\sqrt{\epsilon})$.
For $S$ to be $(\epsilon,\delta)$ stable, we require that $ \|\overline{\Sigma}_S - I\| = O(1) $ with probability $1 - \tau$.
However, the Matrix-Chernoff bound (see, e.g., ~\cite[Theorem 5.1.1]{Tropp15-matrix}) only guarantees that with probability at least $1 - \tau$, $\|\overline{\Sigma}_S - I\| = \tilde{O}(d)$.

The rest of this section is devoted to showing that, with high probability, it is possible to remove $\epsilon' n$ points from $S$ such that both conditions in Claim~\ref{ClaimCovSuffStabMain} are satisfied for the subset.

\subsection{Controlling the Variance} \label{SecBddCovVar}
As a first step, we show that it is possible to remove an $\epsilon$-fraction of points so that the second moment matrix concentrates.
Since finding a subset is a discrete optimization problem, we first perform a continuous relaxation: instead of finding a large subset, we find a suitable distribution on points.
Define the following set of distributions:
\begin{align*}
\Delta_{n , \epsilon} = \Big\{w \in \R^n: 0 \leq w_i \leq 1/((1 - \epsilon)n); \littlesum_{i=1}^n w_i = 1\Big\} \;.
\end{align*}
Note that $\Delta_{n,\epsilon}$ is the convex hull of all the uniform distributions on $S' \subseteq S: |S'| \geq (1 -\epsilon)n $.
In Appendix~\ref{AppCovDetRounding}, we show how to recover a subset $S'$ from the $w$.
Although we use the set $\Delta_{n,\epsilon}$ for the sole purpose of theoretical analysis, the object $\Delta_{n,\epsilon}$ has also been useful in the design of computationally efficient algorithms~\cite{DKKLMS16,DK20-survey}.
We will now show that, with high probability, there exists a $w \in \Delta_{n,\epsilon}$ such that $\overline{\Sigma}_w$ has small spectral norm.

Our proof technique has three main ingredients: (i) minimax duality, (ii) truncation, and (iii)  concentration of truncated empirical processes. 
Let $\cM$ be the set of all PSD matrices with trace norm $1$, i.e., 
$\cM = \left\{ M: M \succeq 0, \trace(M) = 1 \right\}$.
Using minimax duality~\cite{Sion58} and the variational characterization of spectral norm, we obtain the following reformulation:
\begin{align}
\nonumber
\min_{w \in \Delta_{n,\epsilon}}\|\overline{\Sigma}_w - I\|&\leq 1+ \min_{w \in \Delta_{n,\epsilon}}\|\overline{\Sigma}_w \| \\
	&=  1 + \min_{w \in \Delta_{n,\epsilon}} \max_{M \in \cM} \langle \littlesum_{i=1}^n w_i(x_i- \mu)(x_i- \mu)^T ,M \rangle \nonumber
 \\
&= 1+ \max_{M \in \cM} \min_{w \in \Delta_{n,\epsilon}}  \langle \littlesum_{i=1}^n w_i(x_i- \mu)(x_i- \mu)^T , M \rangle.
\label{EqnDualityCov}
\end{align}
This dual reformulation plays a fundamental role in our analysis.
Lemma~\ref{LemBddMatrixProj} below, proved in Appendix~\ref{AppMatProj}, states that, with high probability, all the terms in the dual reformulation are bounded. 
\begin{lemma}
Let $x_1,\dots,x_n$ be $n$ i.i.d. points from a distribution in $\R^d$ with mean $\mu$ and covariance $\Sigma \preceq I$. Let $Q = \Theta(1/ \sqrt{\epsilon} + (1 / \epsilon) \sqrt{ \trace(\Sigma) /n})$.
For $M \in \cM$, let $S_M = \{i \in [n] : (x_i - \mu)^TM(x_i - \mu) \leq Q^2\}$.
Let $\mathcal E$ be the event $\mathcal E = \{  \forall M \in \cM, |S_M| \geq (1 -\epsilon)  n\}$. There exists a constant $c > 0$ such that the event $\mathcal E$ happens with probability at least $1 - \exp( - c\epsilon n)$.
\label{LemBddMatrixProj}
\end{lemma}
Lemma~\ref{LemBddMatrixProj} draws on the results by Lugosi and Mendelson~\cite[Proposition 1]{LugosiM19robust} and Depersin and Lecué~\cite[Proposition 1]{DepLec19}.
The proof is given in Appendix~\ref{AppProjOutlier}.
Importantly, given $n = \Omega( \trace(\Sigma) / \epsilon)$ samples, the threshold $Q$ is $O(1/ \sqrt{\epsilon})$.
Approximating the empirical process in Eq.~\eqref{EqnDualityCov} with a truncated process allows us to use the powerful inequality for concentration of bounded empirical processes due to Talagrand~\cite{Talagrand1996}.
Formally, we show the following lemma:
\begin{lemma} \label{LemStabVariance}
Let $x_1,\dots,x_n$ be $n$ i.i.d. points from a distribution in $\R^d$ with mean $\mu$ 
and covariance $\Sigma \preceq I$. Further assume that for each $i$, $\|x_i - \mu\| = O( \sqrt{\trace(\Sigma) / \epsilon})$. 
 There exists $c, c' > 0$ such that for $\epsilon \in (0, c')$, with probability $ 1 - 2\exp(-c n \epsilon)$, we have that $\min_{w \in \Delta_{n , \epsilon}} \left\| \overline{\Sigma}_w - I\right\| \leq \delta^2/\epsilon$,  where  $\delta = O( \sqrt{( \trace(\Sigma) \log \srank(\Sigma)) / n} + \sqrt{\epsilon})$. 
\end{lemma}
\begin{proof}
Throughout the proof, assume that the event $\mathcal E$ from Lemma~\ref{LemBddMatrixProj} holds. 
Without loss of generality, also assume that $\mu = 0 $.
Let $f: \R_+ \to \R_+$ be the following function:
\begin{align}
f(x) := \begin{cases} x, & \text{ if } x \leq Q^2\\
Q^2, & \text{  otherwise}. 
\end{cases}
\end{align}
It follows directly that $f$ is $1$-Lipschitz and $0 \leq f(x) \leq x$.
Using minimax duality,
\begin{align*}
\min_{w \in \Delta_{n , \epsilon} } \|\overline{\Sigma}_w- I\| &\leq 1 +  \max_{M \in \cM}\min_{w\in \Delta_{n , \epsilon}} \littlesum w_i x_i^TMx_i  \leq 1 + \max_{M \in \cM} \littlesum_{i=1}^n  f(x_i^TMx_i)/((1 - \epsilon)n) ,
\end{align*}
where the second inequality uses that on event $\mathcal{E}$, for every $M \in \cM$, the set $S_M = \{[i] \in n: x_i^TMx_i \leq Q^2 \}$ has cardinality larger than $(1- \epsilon)n$, and thus, the uniform distribution on the set $S_M$ belongs to $\Delta_{n, \epsilon}$.
Define the following empirical processes $R$ and $R'$:
\begin{align*}
R =  \sup_{M \in \cM} \littlesum_{i=1}^n f( x_i^TMx_i), \qquad R' =  \sup_{M \in \cM} \littlesum_{i=1}^n f( x_i^TMx_i) - \E f( x_i^TMx_i).
\end{align*}
As $ 0 \leq f(x) \leq x$, we have that $0 \leq \E f(x_i^TMx) \leq  \E x_i^TMx \leq 1$, which gives that $|R - R'| \leq n$. Overall, we obtain the following bound:
\begin{align*}
\min_{w \in \Delta_{n , \epsilon} } \|\overline{\Sigma}_w- I\| \leq 
1 + R/((1 - \epsilon )n) 
\leq 1 + 2(R' +  n \epsilon)/n \leq (2R')/n + 3.
\end{align*}
Note that $3 \leq \delta^2 / \epsilon$ when $\delta \geq \sqrt{3 \epsilon}$. We now apply Talagrand's concentration inequality on $R'$, as each term is bounded by $Q^2$.
We defer the details to Lemma~\ref{LemTruncVarConcCov} below, showing that $R'/n =  O( \delta^2 /\epsilon )$ with probability $1 - \exp(-c n \epsilon)$.  
By taking a union bound, we get that both $R'/n = O(\delta^2 /\epsilon)$ and $\mathcal{E}$ hold with high probability.
\end{proof}

We provide the details of concentration of the empirical process, related to the variance in Lemma~\ref{LemStabVariance}, which was omitted above.
\begin{lemma}\label{LemTruncVarConcCov}
Consider the setting in the proof of Lemma~\ref{LemStabVariance}. Then, with probability $ 1 - \exp(- n\epsilon)$, $ R' /n  \leq  \delta^2/ \epsilon$,
where $\delta = O(\sqrt{(\trace(\Sigma) \log \srank(\Sigma))/n} + \sqrt{\epsilon})$.
\end{lemma}
\begin{proof}
We will apply Talagrand's concentration inequality for the bounded empirical process, see Theorem~\ref{LemTalagrandBddArbitrarily}.
We first calculate the quantity $\sigma^2$, the wimpy variance, required in Theorem~\ref{LemTalagrandBddArbitrarily} below
\begin{align*}
\sigma^2 &= \sup_{M \in M} \sum_{i=1}^n \Var( f(x_i^TMx_i)) \leq \sup_{M \in M} \sum_{i=1}^n\E (f(x_i^TMx_i))^2\leq  \sup_{M \in M} \sum_{i=1}^n Q^2\E f(x_i^TMx_i)\leq n Q^2,
\end{align*}
where we use that $f(x) \leq Q^2$, $f(x) \leq x$, and $\E x^TMx \leq 1$.
We now focus our attention to $\E R'$.
Let $\xi_i$ be $n$ i.i.d. Rademacher random variables, independent of $x_1,\dots,x_n$.
We use contraction and symmetrization properties for Rademacher averages~\cite{LedouxTalagrand,BGM-book-13} to get
\begin{align}
\E R' &= \E \sup_{M \in \cM} \sum_{i=1}^n f( x_i^TMx_i) - \E f( x_i^TMx_i)  \leq 2\E \sup_{M \in \cM} \sum_{i=1}^n \xi_i f( x_i^TMx_i)
\nonumber
\\
 &\leq 2\E \sup_{M \in \cM} \sum_{i=1}^n  \xi_i x_i^TMx_i = 2\E \| \littlesum_{i=1}^n  \xi_i x_ix_i^T\| 
 =O\left( \sqrt{\frac{n \trace(\Sigma) \log \srank(\Sigma)}{\epsilon}} + \frac{\trace(\Sigma) \log \srank(\Sigma)}{\epsilon}\right), 
\nonumber\end{align}
where the last step uses the refined version of matrix-Bernstein inequality~\cite{Min17-bern}, stated in Theorem~\ref{LemMatrixConc}, with $L = O(\trace(\Sigma)/ \epsilon)$.

Note that the empirical process $R'$ is bounded by $Q^2$.
By applying Talagrand's concentration inequality for bounded empirical processes (Theorem~\ref{LemTalagrandBddArbitrarily}), with probability at least $1 - \exp(- n \epsilon)$, we have 
\begin{align*}
R' &= O \left(\E R' +  \sqrt{n Q^2 }\sqrt{n \epsilon} + Q^2n \epsilon\right) \\
\implies \frac{R'}{n} 
 &= O\left(\frac{\trace(\Sigma) \log \srank(\Sigma)}{n \epsilon} + \sqrt{\frac{\trace(\Sigma) \log \srank(\Sigma)}{n \epsilon}} +  Q \sqrt{\epsilon}+ \epsilon Q^2\right)\\
&= \frac{1}{\epsilon} O\left(\frac{\trace(\Sigma) \log \srank(\Sigma)}{n} + \sqrt{\frac{\trace(\Sigma) \log \srank(\Sigma)}{n }} \sqrt{\epsilon} + Q \epsilon \sqrt{\epsilon} + (\epsilon Q)^2 \right)\\
&= \frac{1}{\epsilon} \left(O\left( \sqrt{\frac{\trace(\Sigma) \log \srank(\Sigma)}{n }} + \sqrt{\epsilon} + \epsilon Q \right)\right)^2\\
&=  \frac{\delta^2}{\epsilon},
\end{align*}
where $\delta = O( \sqrt{\trace(\Sigma) \log \srank(\Sigma) / n } + \sqrt{\epsilon} + \epsilon Q ) = O(\sqrt{\trace(\Sigma) \log \srank(\Sigma) /n } + \sqrt{\epsilon})$, where we use the fact that  $\epsilon Q = O( \sqrt{\epsilon} + \sqrt{\trace(\Sigma)/n} )$.
\end{proof}

\subsection{Controlling the Mean} \label{SecBddCovMean}
Suppose $u^* \in \Delta_{n , \epsilon}$ achieves the minimum in Lemma~\ref{LemStabVariance}, 
i.e., $\|\overline{\Sigma}_{u^*} - I\| \leq \delta^2 /\epsilon$.
It is not necessary that  $\|\mu_{u^*} - \mu\| \leq \delta$.
Recall that our aim is to find a $w \in \Delta_{n , \epsilon}$ that satisfies the conditions: (i) $\|\mu_{w} - \mu\| \leq \delta$, and (ii) $ \|\overline{\Sigma}_w - I\| \leq \delta^2 / \epsilon$.
Given $u^*$, we will remove additional $O(\epsilon)$-fraction of probability mass from $u^*$ to obtain a $w \in \Delta_n$ such that  $\|\mu_w - \mu\| \leq \delta$.
For $u \in \Delta_n$, consider the following set of distributions:
\begin{align*}
\Delta_{n , \epsilon ,u} = \Big\{w: \littlesum_{i=1}^n w_i = 1, w_i \leq u_i/(1 - \epsilon) \Big\} \;.
\end{align*}
For  any $w \in \Delta_{n, \epsilon, u^*}$, we directly obtain that $\Sigma_{w} \preceq \Sigma_{u^*}/(1 - \epsilon) $.
Our main result in this subsection is that, with high probability, there exists a $w^* \in \Delta_{n, 4 \epsilon, u^*}$ such that $\|\mu_{w^*} - \mu\| \leq \delta$.
We first prove an intermediate result, Lemma~\ref{LemCnditionGoodVariance} below, that uses the truncation  (Lemma~\ref{LemBddMatrixProj}) and simplifies the constraint $\Delta_{n, 4 \epsilon, u^*}$. 
Let $g: \R \to \R$ be the following thresholding function: 
\begin{align}
\label{EqnGdef}
g(x) &= \begin{cases} x, & \text{ if } x \in [-Q , Q],\\ 
Q, &\text{ if } x > Q, \\
-Q, &\text{ if } x< -Q. 
\end{cases}
\end{align}
\begin{lemma} Let $w \in \Delta_{n , \epsilon}$ for some $\epsilon \leq 1/2$.
Suppose that the following event $\mathcal{E}$ holds: 
\begin{align*}
\mathcal{E} := \left\{ \sup_{M \in \cM} |\{i : (x_i - \mu)^TM(x_i - \mu) \geq Q^2\}| \leq \epsilon  n\right\}.
\end{align*}
For a unit vector $v$, let $S_v \in [n] $ be the following multiset: $S_v = \{x_i: x_i \in S, |x_i^Tv| \leq Q\}$.
For a unit vector $v$, let $\overline{w}^{(v)}$ be the following distribution:
\begin{align}
\tilde{w}_i^{(v)} &:= \min\left(w_i, \frac{\I\{x_i \in S_v\}}{|S_v|} \right), \qquad \overline{w}^{(v)} := \frac{\tilde{w}^{(v)}}{\|\tilde{w}^{(v)}\|_1}.
\end{align}
Let $g(\cdot)$ be defined as in Eq.~\eqref{EqnGdef}. Then, for all unit vectors $v$, $\overline{w}^{(v)} \in \Delta_{n, 4 \epsilon, w}$. Moreover, the following inequalities hold:
 \begin{align*}
\left	|\sum_{i=1}^n \overline{w}_i^{(v)} v^T(x_i - \mu)\right| \leq 4\epsilon Q + \left|\frac{\sum_{i \in S_v} v^T(x_i - \mu)} {|S_v|} \right|  
\leq 5\epsilon Q + \left| \frac{\sum_{i \in S} g(v^T(x_i - \mu))} {(1 - \epsilon) n} \right|.
	\end{align*}
\label{LemCnditionGoodVariance}
\end{lemma}
\begin{proof}
On the event $\mathcal E$, we have that $|S_v| \geq ( 1 - \epsilon) n$ for all $v \in \cS^{d-1}$.
In order to show that $\overline{w}^{(v)} \in \Delta_{n, 4 \epsilon, w}$, it suffices to show that for all $v$, $\overline{w}_i^{(v)} \leq w_i/(1 - 4 \epsilon)$. By the definition of $\overline{w}_i^{(v)}$, it is sufficient to show that $ \|\tilde{w}^{(v)}\|_1 \geq 1 - 4\epsilon $.
Let $u_{S}$ and $u_{S_v}$ denote the uniform distributions on the multi-sets $S$ and $S_v$ respectively.
Let $d_{\text{TV}}(p,q)$ denote the total variation distance between the distributions $p$ and $q$.
First note that
\begin{align}
\label{EqnTVUpperBOund}
d_{\text{TV}} (w, u_{S_v}) &\leq d_{\text{TV}}(w, u_{S}) + d_{\text{TV}}(u_{S}, u_{S_v}) \leq \frac{\epsilon}{1 - \epsilon} + \frac{\epsilon}{1 - \epsilon} \leq \frac{2 \epsilon}{1 - \epsilon} \leq 4 \epsilon.
\end{align}
We now use the alternative characterization of total variation distance 
(see, e.g., \cite[Lemma 2.1]{Tsybakov08}):
\begin{align*}
 d_{\text{TV}}(p,q) = (1/2) \sum_{i=1}^n |p_i-q_i| = 1 - \sum_{i=1}^n \min(p_i,q_i).
 \end{align*}
Observe that  $\tilde{w}^{(v)} = \min(w,u_{S_v})$; combining this observation with Eq.~\eqref{EqnTVUpperBOund}, we get the following lower bound on $\|\tilde{w}^{(v)}\|_1$:
 \begin{align*}
\|\tilde{w}^{(v)}\|_1 = 1 - d_{\text{TV}}(w, u_{S_v}) \geq 1 -  4 \epsilon.
\end{align*}
This concludes that $\overline{w}^{(v)} \in \Delta_{n, 4 \epsilon, w}$.
We now focus our attention on the second result in the theorem statement.
The first inequality follows from the fact that both distributions $\overline{w}^{(v)}$ and $u_{S_v}$ have total variation distance less than $4 \epsilon$, and supported on $[-Q,Q]$. The second inequality follows from the fact that (i) $|S_v| \geq (1 - \epsilon)n$, (ii) $g(\cdot)$ is identity on $S_v$, and bounded by $Q$ outside $[-Q,Q]$, and (iii) at most $\epsilon$-fraction of the points are outside $S_v$.  
This completes the proof.
\end{proof}

Using Lemma~\ref{LemCnditionGoodVariance}, we prove the following:
\begin{lemma}\label{LemStabMean}
Let $x_1,\dots,x_n$ be $n$ i.i.d. points from a distribution in $\R^d$ with mean $\mu$ 
and covariance $\Sigma \preceq I$. 
Let $0 < \epsilon < 1/2$ and $ u \in \Delta_{n,\epsilon}$.
Then, for a constant $c>0$, the following holds with probability $ 1 - \exp(- c n \epsilon)$:
$\min_{w \in \Delta_{n , 4\epsilon, u}} \| \mu_w - \mu\| \leq \delta, \text{ where } \delta = O\left( \sqrt{ \epsilon } + \sqrt{\trace(\Sigma)/n} \right)$.
\end{lemma}

At a high-level, the proof of Lemma~\ref{LemStabMean} proceeds as follows:
We use duality and the variational characterization of the $\ell_2$ norm to reduce our problem to an empirical process over projections.
We then use Lemma~\ref{LemCnditionGoodVariance} to simplify the domain 
constraint $\Delta_{n, 4 \epsilon, u^*}$ and obtain a bounded empirical process, 
with an overhead of $O(\epsilon Q) = O(\delta)$.

\begin{proof} (Proof of Lemma~\ref{LemStabMean})
Let $\Delta$ be the set $\Delta_{n, 4\epsilon, u}$ and assume that $\mu=0$ without loss of generality. On the event $\mathcal{E}$ (defined in Lemma~\ref{LemBddMatrixProj}), using minimax duality and Claim~\ref{LemCnditionGoodVariance}, we get
\begin{align}
\min_{w \in \Delta} \max_{v \in \cS^{d-1}} \littlesum_{i=1}^n w_i x_i ^Tv  =  \max_{v \in \cS^{d-1}} \min_{w \in \Delta}  \littlesum_{i=1}^n w_i x_i ^Tv \leq 5\epsilon Q +  \max_{v \in \cS^{d-1}}|\littlesum_{i \in[n]} 2g(v^Tx_i)/n|.
\label{EqnControlMean} 
 \end{align}
We define the following empirical processes:
\begin{align*}
N = \sup_{v \in \cS^{d-1}}\littlesum_{i=1}^n g(v^Tx_i), \qquad N' = \sup_{v \in \cS^{d-1}}\littlesum_{i=1}^n g(v^Tx_i) - \E [g(v^Tx_i)].
\end{align*}
As $g(\cdot)$ is an odd function and $\cS^{d-1}$ is an even set, we get that both $N$ and $N'$ are non-negative. 
For any $v \in \cS^{d-1}$, note that $v^Tx$ has variance at most $1$ and $\bP(|v^Tx| \geq Q) = O(\epsilon)$. We can thus bound $\E g(v^Tx)$ as $O (\sqrt{\epsilon}) =  O(\epsilon Q)$ (see Proposition~\ref{PropMeanShiftTruncBddCov}). 
This gives us that $|N - N'| = O(n \epsilon Q)$. 
Using the variational form of the $\ell_2$ norm with Eq.~\eqref{EqnControlMean} leads to the following inequality in terms of $N'$:
\begin{align*}
 \min_{w \in \Delta}\| \mu_w\| = 
\max_{v \in \cS^{d-1}} \min_{w \in \Delta}  \littlesum_{i=1}^n w_i x_i ^Tv \leq 5\epsilon Q  +  N/((1 - \epsilon)n) = O(\epsilon \, Q) + (2N')/n \;.
 \end{align*}
Note that the term $ \epsilon Q$ is small as $ \epsilon Q = O( \delta )$.
As $N'$ is a bounded empirical process, with the bound $Q$, we can apply  Talagrand's concentration inequality.
We defer the details to Lemma~\ref{LemStabMeanEmp} below, showing that $N'/n = O(\sqrt{\trace(\Sigma)/n} + \sqrt{\epsilon}) = O(\delta)$. 
Taking a union bound over concentration of $N'$ and the event $\mathcal E$, we get that the desired result holds with high probability.
\end{proof}

\begin{lemma}\label{LemStabMeanEmp}
Consider the setting in Lemma~\ref{LemStabMean}. Then, with probability, $1 - \exp(- n \epsilon)$, $R'/n = O(\sqrt{\trace(\Sigma)/n} +  \sqrt{\epsilon})$. 
\end{lemma}
\begin{proof}
We will use Talagrand's concentration inequality for bounded empirical processes, stated in Theorem~\ref{LemTalagrandBddArbitrarily}.
We first calculate the wimpy variance required for Theorem~\ref{LemTalagrandBddArbitrarily},
\begin{align}
\sigma^2 &= \sup_{v \in \cS^{d-1}} \sum_{i=1}^n \Var( g(x_i^Tv)) 
		\leq \sup_{v \in \cS^{d-1}} \sum_{i=1}^n\E g(v^Tx_i)^2  \leq \sup_{v \in \cS^{d-1}} n  \E (v^Tx_i)^2  \leq n.
\end{align}
We also bound the quantity $\E R'$ using symmetrization and contraction~\cite{LedouxTalagrand,BGM-book-13} properties of Rademacher averages. We have that
 \begin{align*}
\E R' &=  \E \sup_{v \in \cS^{d-1}} \sum_{i=1}^n  g (v^T x_i) - \E g(v^Tx_i) \leq 2 \E \sup_{v \in \cS^{d-1}} \sum_{i=1}^n \epsilon_i  g (v^T x_i)\\
&\leq 2 \E \sup_{v \in \cS^{d-1}} \sum_{i=1}^n \epsilon_i  v^T x_i = 2 \E \|\sum_{i=1}^n \epsilon_i   x_i \| \leq 2 \sqrt{n \trace(\Sigma)},
 \end{align*}
 where the last step uses that $\epsilon_ix_i$ has covariance $\Sigma$.
By applying Talagrand's concentration inequality for bounded empirical processes (Theorem~\ref{LemTalagrandBddArbitrarily}), we get that with probability at least $1 - \exp(- n \epsilon)$,
\begin{align*}
R'/n  &= O(\E R'/n  + \sqrt{n \epsilon} + Q \epsilon ) = O( \sqrt{ \trace(\Sigma)/n} + \sqrt{\epsilon} ).
 \end{align*}
\end{proof}

\subsection{Proof of Theorem~\ref{ThmStabilityBddCov}}
\label{SecProofThmBddCov}

We first state a result stating that deterministic rounding of weights suffice, proved in Appendix~\ref{AppCovDetRounding}.
\begin{lemma}
For $\epsilon \leq \frac{1}{3}$, let $w \in \Delta_{n,\epsilon}$ be such that for $\epsilon \leq \delta$, we have (i) $\|\mu_w - \mu\| \leq \delta$ and (ii)  $\| \overline{\Sigma}_w - I\| \leq \delta^2 / \epsilon$.
Then there exists a subset $S_1 \subseteq S$ such that 
\begin{enumerate}
	\item $|S_1| \geq (1 - 2\epsilon)|S|$.
	\item $S_1$ is $(\epsilon',  \delta')$ stable with respect to $\mu$ and $\sigma^2=1$, where $\delta' = O( \delta + \sqrt{\epsilon} + \sqrt{\epsilon'}  )$.
	\end{enumerate}
\label{LemDeterministicRoundingMain}
\end{lemma}

In the following, we combine the results in the previous lemmas to obtain the stability of a subset with high probability.
We first give a proof sketch.
\paragraph{Proof Sketch of Theorem~\ref{ThmStabilityBddCov}} By Lemma~\ref{LemStabVariance}, we get that there exists a $u^* \in \Delta_{n, \epsilon}$ such that $\|\overline{\Sigma}_{u^*}- I\| \leq \delta^2 /\epsilon$.
Applying Lemma~\ref{LemStabMean} with this $u^*$, we get that there exists a $w^* \in \Delta_{n, 4 \epsilon, u^*}$ such that $\|\mu_{u^*} - \mu\| \leq \delta$. 
 $ v^T\overline{\Sigma}_{w^*}v \leq (1/(1 - 4 \epsilon)) v^T\overline{\Sigma}_{u^*}v = O(\delta^2 /\epsilon)$, for small enough $\epsilon$.
To obtain a discrete set, we show that rounding $w^*$ to a discrete set only leads to slightly worse constants.

\medskip

We are now ready to prove our main theorem, which we restate for completeness.

\begin{theorem}(Theorem~\ref{ThmStabilityBddCov}) 
Let $x_1,\dots, x_n$ be $n$ i.i.d. points in $\R^d$ from a distribution with mean $\mu$ and covariance $\Sigma$.
Let $ \epsilon' =  O(\log(1/\tau)/n + \epsilon) \leq c$ for a sufficiently small positive constant $c$.
Then, with probability at least $1 - \tau$, there exists a subset $ S' \subseteq S$ s.t. $ |S|' \geq (1 - \epsilon')n$ and $|S'|$ is $ (C\epsilon', \delta)$-stable with respect to $\mu$ and $\|\Sigma\|$ with $\delta = O(\sqrt{(\srank(\Sigma) \log \srank(\Sigma) )/n} + \sqrt{C\epsilon'})$.
\end{theorem}
\begin{proof}

Note that we can assume without loss of generality that $\mu = 0$ and $\|\Sigma\|=1$, upper bound $\delta$ by $\delta = O(\sqrt{\trace(\Sigma)\log(\srank(\Sigma))/n} + \sqrt{C \epsilon'})$; otherwise, apply the following arguments to the random variable $(x_i - \mu)/ \sqrt{\|\Sigma\|}$ (the result holds trivially if $\|\Sigma\|=0$).

We first prove a simpler version of the theorem for distributions with bounded support.
The reason we make this assumption is to apply the matrix concentration results in Theorem~\ref{LemMatrixConc}.

\paragraph{Base case: Bounded support} Assume that $\|x_i - \mu\| = O(\sqrt{\trace(\Sigma) / \epsilon'})$ almost surely.

Note that the bounded support assumption allows us to apply Lemma~\ref{LemStabVariance}. 
Set $\tilde{\epsilon} = \epsilon'/c'$ for a large constant $c'$ to be determined later.
Let $u^* \in \Delta_{n , \tilde{\epsilon}}$ achieve the minimum in Lemma~\ref{LemStabVariance}.
For this $u^*$, let $w^* \in \Delta_{n, 4\tilde{\epsilon}, u^*}$ be the distribution achieving the minimum in Lemma~\ref{LemStabMean}.
Note that the probability of error is at most $2\exp(-  \Omega(n \tilde{\epsilon}))$.
We can choose $\epsilon'$ large enough, $\tilde{\epsilon} = \epsilon'/c = \Omega(\log(1/ \tau)/n)$, so that the probability of failure is at most $1 - \tau$.
Let $\delta = C\sqrt{\trace(\Sigma) \log \srank(\Sigma)/n} + C \sqrt{\tilde{\epsilon}}$ for a large enough constant $C$ to be determined later.
We first look at the variance of $w^*$ using the guarantee of $u^*$ in Lemma~\ref{LemStabVariance}:
\begin{align}
\sum_{i=1}^n w^*_i x_ix_i^T   \preceq \sum_{i=1}^n \frac{1}{1 -\epsilon'}u^*_i x_ix_i^T \preceq 2 \sum_{i=1}^n u^*_i x_ix_i^T \leq \frac{1}{\tilde{\epsilon}}(C \sqrt{ \trace(\Sigma) \log \srank(\Sigma)/n} + C\sqrt{\tilde{\epsilon}})^2.
\end{align}
By choosing $C$ to be a large enough constant, we get that 
$\|\sum_{i=1}^n w^* x_ix_i^T -I \| \leq \delta^2/ \tilde{\epsilon}$.
Now, we look at the mean. Lemma~\ref{LemStabMean} states that
\begin{align}
\left\|\sum_{i=1}^n w^* x_i \right\|   =O\left(\sqrt{\tilde{\epsilon}} + C \sqrt{\frac{\trace(\Sigma)}{n}}\right) \leq \delta.
\end{align}
Since $w^* \in \Delta_{n, 4 	\tilde{\epsilon}, u*} $ and $u^* \in \Delta_{n, \tilde{\epsilon}}$, we have that $w^* \in \Delta_{n, 5\tilde{\epsilon}}$.
Therefore, we have a $w^* \in \Delta_{n, 5 	\tilde{\epsilon}}$ that satisfies the requirements of Lemma~\ref{LemDeterministicRounding}.
Applying Lemma~\ref{LemDeterministicRounding}, we get the desired statement for a set $S' \subseteq S$.
Finally, we can choose the constant $c'$ in the definition of $\tilde{\epsilon}$ large enough, so that the set has cardinality $|S'| \geq (1 - \epsilon')n$. This completes the proof for the case of bounded support. 

\paragraph{General case}
\medskip
We first do a simple truncation. For a large enough constant $C'$, let $E$ be the following event:
\begin{align}
E = \left\{ X: \|X- \mu\| \leq  C'\sqrt{\frac{\trace(\Sigma)}{\epsilon'}}\right\}.
\end{align}
Let $Q$ be the distribution of $X$ conditioned on $E$.
Note that $P$ can be written as a convex combination of two distributions: $Q$ and some distribution $R$, 
\begin{align}
P = (1 - \bP(E)) Q + \bP(E^c) R.
\end{align}
Let $Z \sim Q$. By Chebyshev's inequality, we get that $\bP(E^c) \leq \epsilon'/C'^2$.
Using Lemma~\ref{LemDistPropAfterTrunc}, we get that $\|\E Z - \mu\| = O( \sqrt{\epsilon'})$ and $ \text{Cov}(Z) \preceq I$.
The distribution $Q$ satisfies the assumptions of the base case analyzed above.
Let $S_E$ be the set $\{i: x_i \in E\}$ and let $E_1$ be the following event:
\begin{align}
E_1 = \{ |S_E| \geq (1 - \epsilon'/2)n\}.
 \end{align}
A Chernoff bound implies that given $n$ samples from $P$, for a $c>0$, 
with probability at least $1 - \exp(- c n \epsilon' / C'^2) \geq 1 - \tau/2$ 
(by choosing $C'$ large enough and $ \epsilon' = \Omega(\log(1 / \tau)/n)$),  $E_1$ holds.

For a fixed $m \geq (1 - \epsilon'/2)n$, let $z_1,\dots,z_m$ be $m$ i.i.d. draws from the distribution $Q$.
Applying the theorem statement of the base case for each such $m$, we get that, 
except with probability $\tau/2$, there exists an $S' \subseteq [m] \subseteq [n]$ with 
$|S'| \geq (1 - \epsilon'/2)m \geq (1 - \epsilon'/2)^2n \geq (1 - \epsilon')n$,
 such that $|S'|$ is $(C \epsilon', O(\sqrt{d \log d / n} + \sqrt{C\epsilon'}))$-stable.

As mentioned above (event $E_1$), $m \geq (1 - \epsilon'/2)n$ with probability at least $1 - \tau/2$.
We can now marginalize over $m$ to say that with probability at least $1 -  \tau$, 
there exists a $(C \epsilon', \delta)$ stable set $S'$ of cardinality at least $(1 - \epsilon')n$.

However, we are still not done. We have the guarantee that $S'$ is  stable with respect to $\E Z$.
Using the triangle inequality and Cauchy-Schwarz, we get that the set is also $(C\epsilon',\delta')$ stable 
with respect to $\mu$ as well, where $\delta' = \delta + \|\mu - \E Z\| = \delta + O(\sqrt{\epsilon'})$.
This completes the proof.
\end{proof}

\section{Robust Mean Estimation using Median-of-Means Principle}
\label{sec:bdd_cov_med_of_mean}

In this section, we again consider distributions with finite covariance matrix $\Sigma$.
We now turn our attention to the proof of Theorem~\ref{ThmBddCovMom} that removes the additional logarithmic factor  $\sqrt{\log(r(\Sigma))}$.
In Section~\ref{AppMomStabIID}, we show a result stating that pre-processing on i.i.d. points yields a set that contains a large stable subset (after rescaling).
Then, in Section~\ref{AppMomStabCont}, we use a coupling argument to show a similar result in the strong contamination model.

We recall the median of means principle. Let $k \in [n]$.
\begin{enumerate}
\item  First randomly bucket the data into $k$ disjoint buckets of equal size 
(if $k$ does not divide $n$, remove some samples) and compute their empirical means $z_1,\dots,z_k$.
\item Output (appropriately defined) multivariate median of $z_1,\dots,z_k$.  
 \end{enumerate}

\subsection{Stability of Uncorrupted Data}
\label{AppMomStabIID}
We first recall the result (with different constants) from Depersin and Lecu{\'e}~\cite{DepLec19} in a slightly different notation.
\begin{theorem}\cite[Proposition 1]{DepLec19} \label{ThmMOMDL}
Let $z_1,\dots,z_k$ be $k$ points in $\R^d$ obtained by the median-of-means preprocessing 
on $n$ i.i.d. data $x_1,\dots,x_n$  from a distribution with mean $\mu$ and covariance $\Sigma$.
Let $\cM$ be the set of PSD matrices with trace at most $1$.
Then, there exists a constant $c>0$, such that with probability at least $1 - \exp(-ck)$, 
we have that for all $M \in \cM$, $\left|\{ i \in [k] :  (z_i- \mu)^T M (z_i - \mu) > ( k\|\Sigma\|/n)\delta^2\} \right| \leq \frac{k}{100}$,
where $\delta = O(\sqrt{\srank(\Sigma)/k} + 1)$.
\end{theorem}

We now state our main result in this section, proved using minimax duality, that 
Theorem~\ref{ThmMOMDL} implies stability. We first consider the case of i.i.d. data points, 
as it conveys the underlying idea clearly.
\begin{theorem} \label{AppThmMoMstabIID}
Let $x_1,\dots,x_n$ be $n$ i.i.d. random variables from a distribution with mean $\mu$ and covariance $\Sigma \preceq I$.
For $k \in [n]$,
let $z_1,\dots,z_k$ be the variables obtained by median-of-means preprocessing.
Then, with probability $1 - \exp(-ck)$, where $c$ is a positive universal constant, 
there exists a set $S_1 \subseteq [k]$ and $|S_1| \geq 0.95k$ such that $S_1$ is $(0.1,\delta)$-stable 
with respect to $\mu$ and $k\|\Sigma\|/n$, where $\delta = O(\sqrt{\srank(\Sigma)/n} + 1) $.
\end{theorem}
\begin{proof}
For brevity, let $\sigma = \sqrt{k \|\Sigma\|/n}$.
Suppose that the conclusion in Theorem~\ref{ThmMOMDL} holds with $\delta = O(\sqrt{\srank(\Sigma)/k} + 1)$ such that $\delta \geq 1$, i.e., for every $M \in \cM$, for at least $0.99k$ points $(z_i- \mu)^TM(z_i- \mu	) \leq  \sigma^2\delta^2$.
Using minimax duality, we get that
\begin{align*}
\min_{w \in \Delta_{k, 0.01}} \left\|\sum_{i=1}^k w_i (z_i - \mu)(z_i - \mu)^T\right\| &= \min_{w \in \Delta_{k, 0.01}} \max_{M \in \cM} \langle M, \sum_{i=1}^k w_i (z_i - \mu)(z_i - \mu)^T \rangle \\
&= \max_{M \in \cM}  \min_{w \in \Delta_{k, 0.01}} \langle M, \sum_{i=1}^k w_i (z_i - \mu)(z_i - \mu)^T \rangle\\
&\leq \sigma^2\delta^2,
\end{align*}
where the last step uses the conclusion of Theorem~\ref{ThmMOMDL}.
As $\delta^2 \geq 1$, we also get that $\|\sum_{i=1}^k w_i^* (z_i - \mu)(z_i - \mu)^T - \sigma^2I\|\leq \sigma^2\delta^2$.
Let $w^*$ be the distribution that achieves the minimum in the above statement.
We can also bound the first moment of $w^*$ using the bound on the second moment of $w^*$ as follows:
\begin{align*}
\sum_{i=1}^k w^*_i v^T(z_i - \mu)  \leq \sqrt{\littlesum_{i=1}^k  w^*_i( v^T(z_i - \mu))^2 } \leq \sqrt{\|\littlesum_{i=1} w^*_i (z_i- \mu	)(z_i - \mu)^T\|} \leq \sqrt{ \sigma^2\delta^2} = \sigma\delta.
\end{align*}
Given this $w^* \in \Delta_{k, 0.01}$, we will now obtain a subset of $\{z_1,\dots,z_k\}$ that satisfies the stability condition.
In particular, Lemma~\ref{LemDeterministicRounding} shows that we can deterministically round $w^*$ such that there exists a large stable subset of $\{z_1,\dots,z_k\}$ which is $(0.1, \delta)$ stable with respect to $\mu$ and $\sigma^2$. 
\end{proof}

\subsection{Stability Under Strong Contamination Model}
\label{AppMomStabCont}

We now prove Theorem~\ref{ThmBddCovMom}, i.e., stability of a subset after corruption, using Theorem~\ref{AppThmMoMstabIID}. The following result shares the same principle as \cite[Lemma B.1]{DHL19}: we add a coupling argument because the pre-processing step (random bucketing) introduces an additional source of randomness.
\begin{theorem}(Formal statement of Theorem~\ref{ThmBddCovMom})
Let $T$ be an $\epsilon$-corrupted version of the set $S$, where $S$ is a set of $n$ i.i.d. points from a distribution $P$ with mean $\mu$ and covariance $\Sigma$.
Set $\epsilon' = O(\epsilon + \log(1/ \tau)/n)$ and set $k = \lfloor \epsilon' n\rfloor$.
Let $T_k$ be the set of $k$ points obtained by median-of-means preprocessing on the set $T$.
Then, with probability $1 - \tau$,
$T_k$ is $0.01$-corruption of a set $S_k$ such that there exists a $S_k' \subseteq S_k$, $|S'_k| \geq 0.95k$ and $S_k'$ is $(0.1,\delta)$ stable with respect to $\mu$ and $k \|\Sigma\|/n$, where $\delta = O(\sqrt{\srank(\Sigma)/n} + 1)$. 
   \label{ThmMedOfMeansStabAfterCorr}
\end{theorem}

\begin{proof}

For simplicity, assume $k$ divides $n$ and let $m = n/k$. 

Let $S= \{x_1,\dots,x_n\}$ be the multiset of $n$ i.i.d. points in $\R^d$ from $P$.
We can write $T$ as $T= \{x_1',\dots,x_n'\}$ such that $|\{i: x_i' \neq x_i\}| \leq \epsilon n$.

As the algorithm only gets a multiset, we first order them arbitrarily.  Let $r'_1,\dots,r'_n$ be any arbitrary labelling of points and let $\sigma_1(\cdot)$ be the permutation such that $r'_i = x'_{\sigma_1(i)}$. 
We now split the points randomly into buckets by randomly shuffling them.
  Let $\sigma(\cdot)$ be a uniformly random permutation of $[n]$ independent of $T$ (and $S$).
Define $w_i' = r'_{\sigma(i)} = x'_{\sigma_1(\sigma(i))}$.
For $i \in [k]$, define the bucket $B_i'$ to be the multiset $B_i':= \{w'_{(i-1)m+1}, \dots , w'_{im}\}$.
For $i \in [k]$, define $z_i'$ to be the mean of the set $B_i'$, i.e., $z_i = \mu_{B_i'}$.
That is, the input to the stable algorithm would be the multiset $T_k$, where $T_k = \{z_1',\dots,z_k'\}$.

We now couple the corrupted points with the original points. For $\sigma$ and $\sigma_1$, define their  composition $\sigma'$ as $\sigma'(i) := \sigma_1(\sigma(i))$.
Define $r_i:= x_{\sigma_1(i)}$  and $w_i := r_{\sigma(i)} = x_{\sigma'(i)}$.
Importantly, Proposition~\ref{PropMOMCoupling} below states that $w_i$'s are i.i.d. from $P$.
The analogous bucket for uncorrupted samples is $B_i := \{w_{(i-1)m+1}, \dots , w_{im}\}$.
For $i \in [k]$, define $z_i:= \mu_{B_i}$ and define $S_k$ to be $\{z_1,\dots,z_k\}$.
Therefore, $z_1,\dots,z_k$ are obtained from the median-of-means processing of i.i.d. data $w_{1},\dots,w_{n}$, and thus Theorem~\ref{AppThmMoMstabIID} holds\footnote{ If $(x_1,\dots, x_n)$ are i.i.d., then choosing the buckets $B_i= \{x_{(i-1)m},\dots,x_{im}\}$ for $i \in [k]$ preserves independence. In particular, any partition of $k$ sets of equal cardinality that does not depend on the values of $(x_1,\dots,x_n)$ suffices. Therefore, Theorem~\ref{ThmMOMDL} and Theorem~\ref{AppThmMoMstabIID} hold for this bucketing strategy too. }.
That is, there exists $S_k' \subseteq S_k$ that satisfies the desired properties.

It remains to show that $T_k$ is a corruption of $S_k$.
It is easy to see that $|T_k \cap S_k| \geq k - \epsilon n \geq 0.99k$, by choosing $\epsilon'$ large enough. 
That is, for any $\sigma_1$ and $\sigma$, $T_k$ is at most $(0.01)$-contamination of the set $S_k$.
\end{proof}

\begin{proposition}
Let $x_1,\dots,x_n$ be  $n$ i.i.d. points from a distribution $P$ and $\sigma_1(\cdot)$ be a permutation, potentially depending on $x_1,\dots,x_n$. 
Let $\sigma(\cdot)$ be a random permutation independent of $x_1,\dots, x_n$ and $\sigma_1(\cdot)$.
Define the composition permutation be $\sigma'(i) := \sigma_1(\sigma(i))$.
Then $x_{\sigma'(1)},\dots,x_{\sigma'(n)}$ are also i.i.d. from the distribution $P$.
\label{PropMOMCoupling}
\end{proposition}
\begin{proof}
First observe that $\sigma'(\cdot)$ is a uniform random permutation independent of $x_1,\dots,x_n$.
The result follows from the following fact:
\begin{fact}
Let $x_1,\dots,x_n$ be  $n$ i.i.d. points from a distribution $P$. Let $\sigma(\cdot)$ be a random permutation independent of $x_1,\dots, x_n$, then $x_{\sigma(1)},\dots,x_{\sigma(n)}$ are also i.i.d. from the distribution $P$.
\end{fact}
\end{proof}

\section{Robust Mean Estimation Under Finite Central Moments} %
\label{sec:bounded_central_moments}

In this section, we consider distributions with identity covariance and bounded central moments.
Our main result in this section is the proof of Theorem~\ref{ThmStabilityHighMom}, which obtains a tighter dependence on $\epsilon$.
Our proof strategy closely follows the proof structure of the bounded covariance case.
We suggest the reader to read Section~\ref{sec:bounded_covariance} before reading this section.
This section has a similar organization to Section~\ref{sec:bounded_covariance}.
We start with a simplified stability condition in Lemma~\ref{LemBddMatrixProjMom}.
Sections~\ref{SecHighVarUpp} and~\ref{SecHighVarLow} contain the arguments for controlling the second moment matrix from above and below respectively.
Section~\ref{SecHighMean} contains the results regarding the concentration results for controlling the sample mean.
Finally, we combine the results of the previous sections in Section~\ref{AppHighProofThm} to complete the proof of Theorem~\ref{ThmStabilityHighMom}.

In the bounded covariance setting, we considered $\delta$ such that $\delta = \Omega( \sqrt{\epsilon})$.
As such, we only needed an upper bound on second moment matrix, $\overline{\Sigma}_{S'}$, for a set $S' \subseteq S$ (For $\delta \geq \sqrt{\epsilon}$, the lower bound in the second condition of stability is trivial).
For $\delta= o(\sqrt{\epsilon})$, we need a \textit{sharp} lower bound on the minimum eigenvalue of $\overline{\Sigma}_{S_1}$ for \textit{all large  subsets} $S_1$  of a set $S'$.
Such a result is not possible in general, unless we impose both: (i) identity covariance and (ii) tighter control on tails of $X$.

We will prove the existence of a stable set with high probability using the following claim. This is analogous to Claim~\ref{ClaimCovSuffStabMain} in the bounded covariance setting.
In particular, we also need a lower bound on the minimum eigenvalue of $\overline{\Sigma}_{S'}$ for all large subsets $S'$.
\begin{claim} Let $0 \leq \epsilon \leq \delta$ and  $\epsilon \leq 0.5$.  A set $S$ is $(\epsilon, O(\delta))$ stable with respect to $\mu$ and $\sigma^2 = 1$, if it satisfies the following for all unit vectors $v$.
\begin{enumerate}
	\item $\|\mu_{S} - \mu\| \leq \delta$.
	\item $ v^T \overline{\Sigma}_S v \leq 1 + \delta^2 /\epsilon$.
	\item For all subsets $S' \subseteq S : |S'| \geq (1 -\epsilon) |S|$, $ v^T \overline{\Sigma}_{S'} v \geq (1 - \delta^2 / \epsilon)$.
\end{enumerate}
\label{ClaimSuffStabHighMom}
\end{claim}
The proof of Claim~\ref{ClaimSuffStabHighMom} is provided in Appendix~\ref{AppHighSuff}.

\subsection{Upper Bound on the Second Moment Matrix}
\label{SecHighVarUpp}

For simplicity, we will state our probabilistic results directly in terms of $d$ instead of $\trace(\Sigma)$ and $\srank(\Sigma)$.
The proof techniques of Section~\ref{sec:bounded_covariance} can directly be translated to obtain results in terms of $\Sigma$.
We follow the same strategy as in Section~\ref{SecBddCovVar}.
We first refine the bound on the truncation threshold in the following result, proved in Appendix~\ref{AppMatProj}.
\begin{lemma}
Consider the setting in Theorem~\ref{ThmStabilityHighMom}. Let $Q_k = \Theta(\sigma_k\epsilon^{ -1/k} + (1/ \epsilon)\sqrt{\trace(\Sigma)/n})$.
For each $M \in \cM$, let $S_M$ be the set $\{i: (x_i- \mu)^TM(x_i- \mu) \geq Q_k^2\}$.
Let $E$ be the event $E= \{ \sup_{M \in \cM} |S_M| \leq \epsilon n\}$.
 Then for a $c > 0$, with probability at least $1 - \exp( - c\epsilon n)$,  event $E$ holds.
\label{LemBddMatrixProjMom}
\end{lemma}

We first find a subset such that its covariance matrix is bounded.
For technical reasons, we do not assume that the covariance is exactly identity and allow some slack.
The argument is similar to Lemma~\ref{LemStabVariance} for the bounded covariance.
We also impose some additional constraints to simplify the expression, as those regimes would not hold anyway in the proof.  
\begin{lemma} Let $x_1,\dots,x_n$ be $n$ i.i.d. points in $\R^d$ from a distribution with mean $\mu$, covariance $\Sigma$, and for a $k \geq 4$, the $k$-th central moment is bounded by $\sigma_k$. Further assume that for $\epsilon < 0.5$, covariance matrix $\Sigma$ satisfies that  $ (1 - 2\sigma_k^2 \epsilon^{1 - \frac{2}{k}}) \preceq \Sigma \preceq I$.
 Further assume the following conditions hold:
\begin{enumerate}
\item  $\log(1 / \tau)/n  = O(\epsilon)$. 
	\item  $\|x_i - \mu\|  = O(\sigma_k \sqrt{d} \epsilon^{ -1/k})$ almost surely.
\item $\sigma_k \epsilon^{\frac{1}{2}	- \frac{1}{k}} = O(1)$.
 \end{enumerate}
Then, for a $c > 0$, with probability $ 1 - \tau - \exp(- c n \epsilon)$: $\min_{w \in \Delta_{n , \epsilon}}\left\| \overline{\Sigma}_w \right\| \leq 1 + \delta^2/ \epsilon$, 
where $\delta = O( \sqrt{(d \log d)/ n} + \sigma_k \epsilon^{1 - \frac{1}{k}} + \sigma_4 \sqrt{\log(1 / \tau)/ n} )$.
\label{LemStabVarianceMoment}
\end{lemma}
\begin{proof}
We will assume without loss of generality that $\mu = 0$. We will assume that the event $\mathcal E$ in Lemma~\ref{LemBddMatrixProjMom} holds as it only incurs an additional probability of error of $ \exp(- c n \epsilon)$.
We use the variational characterization of spectral norm and minimax duality to write the following:
\begin{align*}
\min_{w \in \Delta_{n , \epsilon} } \|\sum_i w_i x_ix_i^T\| &= 
\min_{w\in \Delta_{n , \epsilon}} \max_{M \in \cM} \sum w_i\langle  x_ix_i^T, M\rangle\\
&=  \max_{M \in \cM}\min_{w\in \Delta_{n , \epsilon}} \sum w_i x_i^TMx_i   \\
&\leq \max_{M \in \cM} \sum_{i=1}^n \frac{1}{(1 - \epsilon)n} (x_i^TMx_i) \I_{x_i^TMx_i \leq Q_k^2}  ,
\end{align*}
where the third inequality uses Lemma~\ref{LemBddMatrixProjMom}, where it chooses the uniform distribution on the set $S_M = \{x_i : x_i^TMx_i \leq Q_k^2\}$.
Let $f: \R_+ \to \R_+$ be the following function:
\begin{align*}
f(x) := \begin{cases} x, & \text{ if } x \leq Q_k^2\\
Q_k^2, & \text{  otherwise}. 
\end{cases}
\end{align*}
Define the following random variables $R$ and $R'$:
\begin{align*}
R =  \sup_{M \in \cM} \sum_{i=1}^n f( x_i^TMx_i), \qquad R' =  \sup_{M \in \cM} \sum_{i=1}^n f( x_i^TMx_i) - \E f( x_i^TMx_i) .
\end{align*}
By Lemma~\ref{LemmaMomentBound}, we get that $
|\E f(x_i^TMx) - 1| \leq 2\sigma_k^2 \epsilon^{1 - \frac{2}{k}}$, which gives that
\begin{align*}
|R - n- R'| \leq 2n \sigma_k^2 \epsilon^{1 - \frac{2}{k}}.
\end{align*}
We therefore get that
\begin{align*}
\min_{w \in \Delta_{n , \epsilon} } \|\sum_i w_i x_ix_i^T\| -1 
&\leq \max_{M \in \cM} \sum_{i=1}^n \frac{1}{(1 - \epsilon)n} (x_i^TMx_i) \I_{x_i^TMx_i \leq Q_k^2}  - 1\\
&\leq \max_{M \in \cM} \sum_{i=1}^n \frac{1}{(1 - \epsilon)n} f(x_i^TMx_i)  - 1\\
&= \frac{1}{(1 - \epsilon )n } R - 1\\
&\leq \frac{2R'}{n} + 4 \sigma_k^2 \epsilon^{1 - \frac{2}{k}} +  2\epsilon.
\end{align*}
Observe that the last two terms in the above expression are small, i.e., $ \sigma_k^2 \epsilon^{1 - \frac{1}{k}} + \epsilon = O(\delta^2 / \epsilon)$. 
We next use Lemma~\ref{LemTruncVarEmpMoment} in Appendix to conclude that $R'$ concentrates well. Lemma~\ref{LemTruncVarEmpMoment} states that with probability $1 - \tau$, 
$R'/n \leq (1 / \epsilon)(O( \sqrt{d \log d / n} + \sigma_k \epsilon^{1 - \frac{1}{k}} + \sigma_4 \sqrt{\log(1/\tau)/n} ))^2 $. 
Note that both of the remaining terms are small compared to  
Overall, we get that 
\begin{align*}
\min_{w \in \Delta_{n,\epsilon}}\| \overline{\Sigma}_w \| \leq 1 + \frac{\delta^2}{\epsilon} \;.
 \end{align*}
Taking a union bound on the event $\mathcal E$ and concentration of $R'$ concludes the result.
\end{proof}

\subsection{Minimum Eigenvalue of Large Subsets} \label{SecHighVarLow}

In this section, we prove that under bounded central moments, the minimum eigenvalue of  $\Sigma_{S'}$, of each large enough subset $S'$, has a lower bound close to $1$.
Our result is similar in spirit to Koltchinskii and Mendelson~\cite[Theorem 1.3]{KolMend15} that only  bounds the eigenvalue of $\overline{\Sigma}_S$.
The proof of the following lemma is very similar to the proof of Lemma~\ref{LemStabVarianceMoment}.
 
\begin{lemma} Consider the setting in Lemma~\ref{LemStabVarianceMoment}.
Then, for a constant $c > 0$, with probability $ 1 - \tau - \exp(- c n \epsilon)$, the following holds:
\begin{align*}
\min_{S': |S'| \geq (1 - \epsilon) n} v^T \overline{\Sigma}_{S'}v  \geq 1 - \frac{\delta^2}{\epsilon},
\end{align*}
where $ \delta = O( \sqrt{\frac{d \log d}{n}} + \sigma_k \epsilon^{1 - \frac{1}{k}}	  + \sigma_4\sqrt{\frac{\log( \frac{1}{\tau}) } {n}} )$.
\label{LemMinimumVarianceFormal}
\end{lemma}
\begin{proof}
Without loss of generality, assume that $\mu=0$.
We will assume that event $\mathcal E$ from Lemma~\ref{LemBddMatrixProjMom} holds, with an additional probability of error $\exp(- c n \epsilon)$, that is
\begin{align*}
 \sup_{v \in \cS^{d-1}}  \left|\left\{ i: x_i^Tv \geq Q_{k}   \right\} \right|  \leq  n \epsilon.
\end{align*}
Let $f$ be as defined in the proof of Lemma~\ref{LemStabVarianceMoment}.
For a sequence $y_1,\dots, y_n$, let $y_{(1)},\dots, y_{(n)}$ be its rearrangement in non-decreasing order.
For any unit vector $v$, we have that 
\begin{align*}
\min_{S': |S'| \geq (1 - \epsilon) n} v^T \overline{\Sigma}_{S'}v \geq \min_{w \in \Delta_{n , \epsilon}}  v^T \overline{\Sigma}_wv &= \min_{w \in \Delta_{n , \epsilon}} \sum_{i=1}^n w_i (x_i^Tv)^2\\ 
&\geq \sum_{i=1}^{(1 - \epsilon)n}  (x_i^Tv)_{(i)}^2/ ( (1 - \epsilon)n )  \\
&\geq \sum_{i=1}^{n}  (f((x_i^Tv)^2) - Q_{k}^2 \epsilon n)/( (1 - \epsilon)n ), 
\end{align*}
where we use that at most $\epsilon n$ points have projections larger than $Q_k^2$.
Thus we get that the minimum eigenvalue of any large subset is lower bounded by:
\begin{align*}
\min_{w \in \Delta_{n , \epsilon}} \min_{v \in \cS^{d-1}} \sum_{i=1}^n w_i (x_i^Tv)^2  \geq \min_{v \in \cS^{d-1}} \sum_{i=1}^{n}  f((x_i^Tv)^2) - Q_{k}^2 \epsilon n.
\end{align*}
Let $h(\cdot)$ be the negative of the function $f(\cdot)$.
Define the  following random variable $Z$ and its counterpart $Z'$:
\begin{align*}
Z &:= \sup_{ v \in \cS^{d-1}} \sum_{i=1}^n h( (x_i^Tv)^2  ), \qquad
Z' := \sup_{ v \in \cS^{d-1}} \sum_{i=1}^n h( (x_i^Tv)^2  ) - \E h ((x_i^Tv)^2) 
\end{align*}
From Lemma~\ref{LemmaMomentBound}, it follows that $|\E h((x_i^Tv)^2) + 1| = | \E f((x_i^Tv)^2) - 1| =O(\sigma_k^2 \epsilon^{1 - \frac{2}{k}})$. 
This immediately gives us that
\begin{align*}
|Z ' -Z - n| = O(n \sigma_k^2 \epsilon ^{1 - \frac{2}{k}}).
\end{align*}
Therefore, the desired quantity satisfies the following inequalities:
\begin{align*}
(1 - \epsilon ) n \min_{w \in \Delta_{n , \epsilon}} \min_{v \in \cS^{d-1}} \sum_{i=1}^n w_i (x_i^Tv)^2  &\geq \min_{v \in \cS^{d-1}} \sum_{i=1}^{n}  f((x_i^Tv)^2) - Q_{k}^2 \epsilon n\\
&= - \sup_{v \in \cS^{d-1}} \sum_{i=1}^{n}  h((x_i^Tv)^2) - Q_{k}^2 \epsilon n \\
&= - Z - Q_{k}^2 \epsilon n \\
&\geq -Z' + n - O(n \sigma_k^2 \epsilon^{1 - \frac{2}{k}}) - \epsilon Q_k^2 \epsilon n. \end{align*}
We thus require a high probability upper bound on $Z'$.
Note that $Z'$ behaves similarly to $R'$, defined in the proof of Lemma~\ref{LemStabVarianceMoment}.
Similar to the proof of Lemma~\ref{LemTruncVarEmpMoment}, we get that, with probability at least $1 - \tau$,
\begin{align*}
\frac{Z'}{n}   &\leq \frac{1}{\epsilon} \Big(O\Big( \sqrt{\frac{d \log d}{n}} + \sigma_k \epsilon^{1 - \frac{1}{k}} + \sigma_4\sqrt{\frac{\log( 1/\tau)} {n}} \Big) \Big)^2. 
\end{align*}
Note that the remaining terms $ \sigma_k^2 \epsilon^{1 - \frac{2}{k}} = O( \delta^2 / \epsilon)$ and $ \epsilon Q_k^2 = O(\sigma_k^2 \epsilon^{ - \frac{2}{k} +  \frac{d}{n \epsilon} } = O(\delta^2 / \epsilon) $.
Therefore, we get the minimum eigenvalue of any large subset is at least
\begin{align*}
\min_{w \in \Delta_{n, \epsilon}} \lambda_{\min}( \overline{\Sigma}_{w} )  &\geq 1- \frac{\delta^2}{\epsilon},
\end{align*}
where $\delta = O( \sqrt{\frac{d \log d}{n}} + \sigma_k \epsilon^{1 - \frac{1}{k}} +  \sigma_4\sqrt{\frac{\log( \frac{1}{\tau}) } {n}} )$.

\end{proof}
\subsection{Controlling the Mean}
\label{SecHighMean}
Lemmas~\ref{LemStabVarianceMoment} and \ref{LemMinimumVarianceFormal} give a control on the second moment matrix. 
We will now further remove $O(\epsilon)$ fraction of points to obtain $w$ such that $\|\mu_w - \mu\|$ is small. 
\begin{lemma} Let $x_1,\dots,x_n$ be $n$ i.i.d. random variables from a distribution with  mean $\mu$ and covariance $\Sigma \preceq I$. 
Further, assume that the $x_i$'s are drawn from a distribution with $k$-th bounded central moment $\sigma_k$ 
for a $k\geq 4$. 
Let $u \in \Delta_{n , \epsilon}$.
Assume that $\log(1/\tau)/n = O(\epsilon)$.
Then, for a constant $c>0$, the following holds with probability $ 1 - \tau - \exp(- c n \epsilon)$:
\begin{align*}
\min_{w \in \Delta_{n , 4\epsilon, u}} \| \sum_{i=1}^n w_i x_i - \mu\| =O(\sqrt{d/n} + \sigma_k \epsilon^{1 - \frac{1}{k}}    + \sqrt{\log(1/\tau)/n} ).
\end{align*}
\label{LemStabMeanMoment}
\end{lemma}
\begin{proof}
Without loss of generality, let us assume that $\mu = 0$.
Also, assume that the event $\mathcal E$ from Lemma~\ref{LemBddMatrixProjMom} holds, with the additional error of $\exp(- c n \epsilon)$.
Let $g(\cdot)$ be the following function:
\begin{align*}
g(x) &= \begin{cases} x, & \text{ if } x \in [-Q_k , Q_k]\\ 
Q_k, &\text{ if } x > Q_k \\
-Q_k, &\text{ if } x< -Q_k. 
\end{cases}
\end{align*}
Let $N$ be the following random variable:
\begin{align*}
N = \sup_{v \in \cS^{d-1}}\sum_{i=1}^n g(v^Tx_i) = \sup_{v \in \cS^{d-1}}\left|\sum_{i=1}^n g(v^Tx_i) \right|,
\end{align*}
where we use that $g(\cdot)$ is an odd function.
We also define the following empirical process, where each term is centered:
\begin{align*}
N' = \sup_{v \in \cS^{d-1}}\sum_{i=1}^n g(v^Tx_i) - \E [g(v^Tx_i)] = \left|\sup_{v \in \cS^{d-1}}\sum_{i=1}^n g(v^Tx_i) - \E [g(v^Tx_i)] \right|.
\end{align*}
As $Q_k = \Omega(\sigma_k \epsilon^{-1/k})$, Lemma~\ref{LemmaMomentBound} states that $\sup_v \E g(v^Tx) =O(\sigma_k \epsilon^{1 - \frac{1}{k}}) $, and this gives that
\begin{align*}
|N - N'| =  O( n\sigma_k \epsilon^{1 - \frac{1}{k}}).
\end{align*}
We now use duality to write the following:
\begin{align*}
\min_{w \in \Delta_{n , \epsilon, u}} \| \sum_{i=1}^n w_i x_i\|  &= \min_{w \in \Delta_{n , \epsilon, u}} \max_{v \in \cS^{d-1}} \langle \sum_{i=1}^n w_i x_i, v \rangle  \\
&=  \max_{v \in \cS^{d-1}} \min_{w \in \Delta_{n , \epsilon, u}}  \langle \sum_{i=1}^n w_i x_i, v \rangle \\
&\leq 5\epsilon Q_k  +  \left|\frac{1}{(1 - \epsilon)n} N \right| \leq O(\epsilon Q_k) + O(\sigma_k \epsilon^{1 - 1/k}) + 2N',
 \end{align*}
where the last step uses Lemma~\ref{LemCnditionGoodVariance}. 
We now use Lemma~\ref{LemTruncMeanEmpMoment} to conclude that $N'$ concentrates. Recall that $ \epsilon Q_k = O( \sigma_k \epsilon^{1 - \frac{1}{k}} + \sqrt{d/n} )$.
Overall, we get that, with probability $1- \tau - \exp(-n 	\epsilon)$, there exits a $w \in  \Delta_{n , \epsilon, u}$, such that $ \|\sum w_ix_i \|$ = $O(\sqrt{d/n} + \sigma_k \epsilon^{1 - \frac{1}{k}}    + \sqrt{\log(1 / \tau)/n} ).$
\end{proof}

\subsection{Proof of Theorem~\ref{ThmStabilityHighMom}}
\label{AppHighProofThm}

We now combine the results in the previous lemmas to obtain the stability of a subset with high probability. 
Although we prove the following result showing the existence of $(2 \epsilon' ,\delta)$ stable subset, this can generalized to existence of $(C \epsilon, O(\delta))$ stable subset for a large constant $C$.
\begin{theorem} (Theorem~\ref{ThmStabilityHighMom}) Let $S = \{x_1,\dots, x_n\}\subset \R^d$ be $n$ i.i.d. points from a distribution with mean $\mu$ and covariance $\Sigma$ such that $ (1 - 2\sigma_k^2 \gamma^{1 - \frac{1}{k}}) I\preceq  \Sigma \preceq I$. Further assume that for a $k \geq  4$, the $k^{\text{th}}$ central moment is bounded by $\sigma_k$.
Let $ \epsilon' = \Theta(\epsilon + \frac{\log(1/\tau)}{n})  \leq c$ for a sufficiently small constant $c$.

Then, with probability at least $1 - \tau$, there exists a subset $ S' \subseteq S$ s.t. $ |S'| \geq (1 - \epsilon')n$ and $|S'|$ is $ (2\epsilon', \delta)$-stable with $\delta = O(\sigma_k \epsilon^{1 - \frac{1}{k} } + \sqrt{ \frac{d \log d}{n} } + \sigma_4 \sqrt{\frac{\log(1/ \tau)}{n} })$.
\end{theorem}
\begin{proof}
First note that, for the bounded covariance condition, Theorem~\ref{ThmStabilityBddCov} already gives a guarantee that, with probability at least $1 - \tau$,
\begin{align}
\|\widehat{\mu}- \mu\| = O\Big( \sqrt{ (d \log d )/ n} + \sqrt{\epsilon} + \sqrt{\log(1 / \tau)/n}	\Big).
\end{align}
Therefore, the guarantee of this theorem statement is tighter only in the following regimes:
\begin{align}
\log(1/ \tau)/n = O(\epsilon), \qquad O( \sigma_k \epsilon^{\frac{1	}{2} - \frac{1}{k}}) = O(1), \qquad d \log d/n = O(\epsilon).
\end{align}
For the rest of the proof, we will assume that all three of these conditions hold.
Similar to the proof of Theorem~\ref{ThmStabilityBddCov}, we will first prove the statement when the samples are bounded.
Without loss of generality, we will assume $\mu = 0$.

\paragraph{Base case: Bounded support} In this case, we will assume that $\|x_i\| =O(\sigma_k \epsilon^{-1/k} \sqrt{d})$ almost surely.
We will use Lemma~\ref{LemRandRound} to show that the set is stable.
Set $\tilde{\epsilon} = \epsilon'/C'$ for a large enough constant $C'$ to be determined later.

Note that $x_1,\dots,x_n$ satisfy the conditions of Lemmas~\ref{LemMinimumVarianceFormal}, \ref{LemStabVarianceMoment}, and \ref{LemStabMeanMoment}.
In particular,  we will use Lemma~\ref{LemMinimumVarianceFormal} with $ C \tilde{\epsilon}$, where $C$ is large enough. By choosing $\epsilon' = \Omega(\log(1/\tau)/n)$, we get that, with probability $ 1 - \tau/3$, for any $S' : |S'| \geq (1 - C \tilde{\epsilon})n$ and unit vector $v$,
\begin{align}
 \frac{\sum_{i \in S'} (v^Tx_i)^2}{|S'|} \geq 1- \frac{\delta^2}{ C\tilde{\epsilon}}.
\end{align}

We first look at the variance using the guarantee in Lemma~\ref{LemStabVarianceMoment}:
Let $u \in \Delta_{n, \tilde{\epsilon}}$ be the distribution achieving the minimum in Lemma~\ref{LemStabVarianceMoment}. By choosing $\epsilon' = \Omega(\log(1/\tau)/n)$, we get that with probability $1 - \tau/3$,
\begin{align}
\sum_{i=1}^n u_i(x_i^Tv)^2   \leq 1  + \frac{\delta^2}{\tilde{\epsilon}}.
\end{align}
We now obtain a guarantee on the mean using Lemma~\ref{LemStabMeanMoment}. For this $u$, let $w \in \Delta_{n, 4 \tilde{\epsilon}, u}$ be the distribution achieving the minimum in Lemma~\ref{LemStabMeanMoment}. Then with probability $1 - \tau/3$,
\begin{align}
\|\sum_{i=1}^n w_i x_i \|   \leq \delta.
\end{align}
Since $u \in \Delta_{n, 4\tilde{\epsilon}, w} $ and $w \in \Delta_{n, \tilde{\epsilon}, u}$, we have that $u \in \Delta_{n, 5\tilde{\epsilon}}$.
Moreover,
\begin{align}
\sum_{i=1}^n w_i(x_i^Tv)^2   \leq \sum_{i=1}^n \frac{u_i}{1 - \tilde{\epsilon}}(x_i^Tv)   = \frac{1}{1- \tilde{\epsilon}}(1 + \frac{\delta^2}{\tilde{\epsilon}}) \leq 1  + \frac{1}{1 -\tilde{\epsilon}}(\tilde{\epsilon} + \frac{\delta^2}{\tilde{\epsilon}}) \leq 1 + \frac{4 \delta^2}{\tilde{\epsilon}}.
\end{align}
Therefore, we have that $u \in \Delta_{n, 5\tilde{\epsilon}}$ and  satisfies the requirements of  Lemma~\ref{LemRandRound}, where we note that $r_1 = O(1)$ and $r_2 = O(1)$ to get the desired statement.
By a union bound, the failure probability is $\tau$.
Finally, we choose $C$ and $C'$ large enough such that the cardinality of the stable set is at least $(1 - \epsilon')n$ and it is $(2 \epsilon',\delta)$ stable.

\paragraph{General case: Unbounded support}

We first do a simple truncation. Let $E$ be the following event:
\begin{align}
E = \{ X: \|X\| \leq  C \sigma_k \epsilon^{-\frac{1}{k}}\sqrt{d}\}.
\end{align}
Let $Q$ be the distribution of $X$ conditioned on $E$.
Note that $P$ can be written as convex combination of two distributions: $Q$ and some distribution $R$, 
\begin{align}
P = (1 - \bP(E)) Q + \bP(E^c) R.
\end{align}
Let $Z \sim Q$.
Using Lemma~\ref{LemDistPropAfterTrunc}, we get that $\|\E Z \| \leq 2 \sigma_k \epsilon^{1 - \frac{1}{k}}/C^k$ and $ (1 - 3 \sigma_k^2 \epsilon^{1 - \frac{2}{k}}/C^k) \preceq \text{Cov}(Z) \preceq I$.
Thus the distribution $Q$ satisfies the assumptions of the base case for $C \geq 2$.

Let $S_E$ be the set $\{X_i: X_i \in E\}$.
A Chernoff bound gives that given $n$ samples from $P$, with probability at least $1 - \exp(- n \epsilon')$,
\begin{align}
E_1 = \{ |S_E| \geq (1 - \epsilon'/2)n \}.
 \end{align}
For a fixed $m \geq (1 - \epsilon'/2)n$, let $z_1,\dots,z_m$ be $m$ i.i.d. draws from the distribution $Q$.
Applying the theorem statement for $Q$, as it satisfies the base case above, we get that, with probability at least $1 - \exp(-c m \epsilon')$, there $\exists$  $S'\subset [m] : |S'| \geq (1 -  \epsilon'/2)m \geq (1 -  \epsilon'/2)^2n \geq (1 - \epsilon')n$,
 such that $S'$ is $(2\epsilon', \delta')$-stable.
This gives us a set $S'$ which is stable with respect to $\E Z$.
Using triangle inequality, we get that the set $S'$ is $(\epsilon,\delta')$ stable with respect to $\mu$ as well, where $\delta' = \delta + \|\mu - \E Z\| = \delta + O(\sigma_k \epsilon^{1 - \frac{1}{k}})$.

We can now marginalize over $m$ to get that with probability except $1 - 2\exp(-c n \epsilon')$, the desired claim holds.
Choosing $\epsilon' = \Omega(\log(1/\tau)n)$, we can make probability of failure less than $\tau$.
\end{proof}

\section{Conclusions and Open Problems}
In this paper, we showed that a standard stability condition from the recent high-dimensional robust statistics literature 
suffices to obtain near-subgaussian rates for robust mean estimation in the strong contamination model. With a simple
pre-processing (bucketing), this leads to efficient outlier-robust estimators with subgaussian rates 
under only a bounded covariance assumption.  An interesting technical question is whether the extra $\log d$ 
factor in Theorem~\ref{ThmStabilityBddCov} is actually needed. 
(Our results imply that it is not needed when $\epsilon = \Omega(1)$.) 
If not, this would imply that stability-based algorithms achieve subgaussian
rates without the pre-processing.

\bibliographystyle{alphaurl}
\bibliography{allrefs}

\newcommand{\etalchar}[1]{$^{#1}$}
\begin{thebibliography}{DKK{\etalchar{+}}18}

\bibitem[AMS99]{AlonMS99}
N.~Alon, Y.~Matias, and M.~Szegedy.
\newblock The space complexity of approximating the frequency moments.
\newblock {\em J. Comput. Syst. Sci.}, 58(1):137--147, 1999.

\bibitem[BLM13]{BGM-book-13}
S.~Boucheron, G.~Lugosi, and P.~Massart.
\newblock {\em Concentration Inequalities: A Nonasymptotic Theory of
  Independence}.
\newblock {Oxford University Press}, {Oxford New York, NY}, paperback edition,
  2013.

\bibitem[Cat12]{Catoni12}
O.~Catoni.
\newblock Challenging the empirical mean and empirical variance: A deviation
  study.
\newblock {\em Ann. Inst. H. Poincare Probab. Statist.}, 48(4):1148--1185, 11
  2012.

\bibitem[CDG19]{ChengDG18}
Y.~Cheng, I.~Diakonikolas, and R.~Ge.
\newblock High-dimensional robust mean estimation in nearly-linear time.
\newblock In {\em Proc.\ 30th Annual Symposium on Discrete Algorithms (SODA)},
  pages 2755--2771. {SIAM}, 2019.

\bibitem[CDGS20]{CDGS20}
Y.~Cheng, I.~Diakonikolas, R.~Ge, and M.~Soltanolkotabi.
\newblock High-dimensional robust mean estimation via gradient descent.
\newblock In {\em Proc.\ 37th International Conference on Machine Learning
  (ICML)}, 2020.

\bibitem[CFB19]{CFB19}
Y.~Cherapanamjeri, N.~Flammarion, and P.~L. Bartlett.
\newblock Fast mean estimation with sub-gaussian rates.
\newblock In {\em Conference on Learning Theory, {COLT} 2019}, volume~99 of
  {\em Proceedings of Machine Learning Research}, pages 786--806. {PMLR}, 2019.

\bibitem[DHL19]{DHL19}
Y.~Dong, S.~B. Hopkins, and J.~Li.
\newblock Quantum entropy scoring for fast robust mean estimation and improved
  outlier detection.
\newblock {\em CoRR}, abs/1906.11366, 2019.
\newblock Conference version in NeurIPS 2019.
\newblock URL: \url{http://arxiv.org/abs/1906.11366}, \href
  {http://arxiv.org/abs/1906.11366} {\path{arXiv:1906.11366}}.

\bibitem[DK19]{DK20-survey}
I.~Diakonikolas and D.~M. Kane.
\newblock Recent advances in algorithmic high-dimensional robust statistics.
\newblock {\em CoRR}, abs/1911.05911, 2019.
\newblock URL: \url{http://arxiv.org/abs/1911.05911}, \href
  {http://arxiv.org/abs/1911.05911} {\path{arXiv:1911.05911}}.

\bibitem[DKK{\etalchar{+}}16]{DKKLMS16}
I.~Diakonikolas, G.~Kamath, D.~M. Kane, J.~Li, A.~Moitra, and A.~Stewart.
\newblock Robust estimators in high dimensions without the computational
  intractability.
\newblock In {\em Proc.\ 57th IEEE Symposium on Foundations of Computer Science
  (FOCS)}, pages 655--664, 2016.

\bibitem[DKK{\etalchar{+}}17]{DKK+17}
I.~Diakonikolas, G.~Kamath, D.~M. Kane, J.~Li, A.~Moitra, and A.~Stewart.
\newblock Being robust (in high dimensions) can be practical.
\newblock In {\em Proc.\ 34th International Conference on Machine Learning
  (ICML)}, pages 999--1008, 2017.

\bibitem[DKK{\etalchar{+}}18]{DiakonikolasKKLSS2018sever}
I.~Diakonikolas, G.~Kamath, D.~M. Kane, J.~Li, J.~Steinhardt, and A.~Stewart.
\newblock Sever: {A} robust meta-algorithm for stochastic optimization.
\newblock {\em CoRR}, abs/1803.02815, 2018.
\newblock Conference version in ICML 2019.
\newblock URL: \url{http://arxiv.org/abs/1803.02815}, \href
  {http://arxiv.org/abs/1803.02815} {\path{arXiv:1803.02815}}.

\bibitem[DL19]{DepLec19}
J.~Depersin and G.~Lecue.
\newblock Robust subgaussian estimation of a mean vector in nearly linear time.
\newblock {\em CoRR}, abs/1906.03058, 2019.

\bibitem[DLLO16]{Devroye2016}
L.~Devroye, M.~Lerasle, G.~Lugosi, and R.~I. Oliveira.
\newblock Sub-gaussian mean estimators.
\newblock {\em Ann. Statist.}, 44(6):2695--2725, 12 2016.

\bibitem[HL19]{HL19}
S.~B. Hopkins and J.~Li.
\newblock How hard is robust mean estimation?
\newblock In {\em Conference on Learning Theory, {COLT} 2019}, pages
  1649--1682, 2019.

\bibitem[HLZ20]{HopLiZhang20}
S.~B. Hopkins, J.~Li, and F.~Zhang.
\newblock Robust and heavy-tailed mean estimation made simple, via regret
  minimization.
\newblock In {\em Proc. 35th Annual Conference on Neural Information Processing
  Systems (NeurIPS)}, 2020.

\bibitem[Hop20]{Hop18}
S.~B. Hopkins.
\newblock Mean estimation with sub-{{Gaussian}} rates in polynomial time.
\newblock {\em Ann. Statist.}, 48(2):1193--1213, 2020.
\newblock \href {https://doi.org/10.1214/19-AOS1843}
  {\path{doi:10.1214/19-AOS1843}}.

\bibitem[Hub64]{Huber64}
P.~J. Huber.
\newblock Robust estimation of a location parameter.
\newblock {\em Ann. Math. Statist.}, 35(1):73--101, 03 1964.

\bibitem[JVV86]{JVV86}
M.~Jerrum, L.~G. Valiant, and V.~V. Vazirani.
\newblock Random generation of combinatorial structures from a uniform
  distribution.
\newblock {\em Theor. Comput. Sci.}, 43:169--188, 1986.

\bibitem[KM15]{KolMend15}
V.~Koltchinskii and S.~Mendelson.
\newblock {Bounding the Smallest Singular Value of a Random Matrix Without
  Concentration}.
\newblock {\em International Mathematics Research Notices},
  2015(23):12991--13008, 03 2015.

\bibitem[LLVZ20]{LLVZ19}
Z.~Lei, K.~Luh, P.~Venkat, and F.~Zhang.
\newblock A fast spectral algorithm for mean estimation with sub-gaussian
  rates.
\newblock In {\em Proc.\ 33rd Annual Conference on Learning Theory (COLT)},
  2020.

\bibitem[LM19a]{LugosiM19-survey}
G.~Lugosi and S.~Mendelson.
\newblock Mean estimation and regression under heavy-tailed distributions: {A}
  survey.
\newblock {\em Foundations of Computational Mathematics}, 19(5):1145--1190,
  2019.

\bibitem[LM19b]{LM19-aos}
G.~Lugosi and S.~Mendelson.
\newblock Sub-gaussian estimators of the mean of a random vector.
\newblock {\em Ann. Statist.}, 47(2):783--794, 04 2019.

\bibitem[LM21]{LugosiM19robust}
G.~Lugosi and S.~Mendelson.
\newblock Robust multivariate mean estimation: {{The}} optimality of trimmed
  mean.
\newblock {\em Ann. Statist.}, 49(1):393--410, 2021.
\newblock \href {https://doi.org/10.1214/20-AOS1961}
  {\path{doi:10.1214/20-AOS1961}}.

\bibitem[LRV16]{LaiRV16}
K.~A. Lai, A.~B. Rao, and S.~Vempala.
\newblock Agnostic estimation of mean and covariance.
\newblock In {\em Proc.\ 57th IEEE Symposium on Foundations of Computer Science
  (FOCS)}, pages 665--674, 2016.

\bibitem[LT91]{LedouxTalagrand}
M.~Ledoux and M.~Talagrand.
\newblock {\em Probability in Banach Spaces}.
\newblock Springer, 1991.

\bibitem[Min15]{Minsker15}
S.~Minsker.
\newblock Geometric median and robust estimation in {{Banach}} spaces.
\newblock {\em Bernoulli}, 21(4):2308--2335, 2015.

\bibitem[Min17]{Min17-bern}
S.~Minsker.
\newblock On some extensions of {Bernstein}’s inequality for self-adjoint
  operators.
\newblock {\em Statistics \& Probability Letters}, 127:111--119, August 2017.
\newblock \href {https://doi.org/10.1016/j.spl.2017.03.020}
  {\path{doi:10.1016/j.spl.2017.03.020}}.

\bibitem[NU83]{NY83}
A.~S. Nemirovsky and D.B. Udin.
\newblock {\em Problem complexity and method efficiency in optimization}.
\newblock Wiley,, 1983.

\bibitem[PBR19]{PBR19}
A.~Prasad, S.~Balakrishnan, and P.~Ravikumar.
\newblock A unified approach to robust mean estimation.
\newblock {\em CoRR}, abs/1907.00927, 2019.
\newblock URL: \url{http://arxiv.org/abs/1907.00927}, \href
  {http://arxiv.org/abs/1907.00927} {\path{arXiv:1907.00927}}.

\bibitem[PSBR20]{PrasadSBR2018}
A.~Prasad, A.~S. Suggala, S.~Balakrishnan, and P.~Ravikumar.
\newblock Robust estimation via robust gradient estimation.
\newblock {\em Journal of the Royal Statistical Society: Series B (Statistical
  Methodology)}, 82(3):601--627, July 2020.
\newblock \href {https://doi.org/10.1111/rssb.12364}
  {\path{doi:10.1111/rssb.12364}}.

\bibitem[SCV18]{SteinhardtCV18}
J.~Steinhardt, M.~Charikar, and G.~Valiant.
\newblock Resilience: {A} criterion for learning in the presence of arbitrary
  outliers.
\newblock In {\em Proc.\ 9th Innovations in Theoretical Computer Science
  Conference (ITCS)}, pages 45:1--45:21, 2018.

\bibitem[Sio58]{Sion58}
M.~Sion.
\newblock On general minimax theorems.
\newblock {\em Pacific Journal of Mathematics}, 8(1):171--176, 1958.

\bibitem[Tal96]{Talagrand1996}
M.~Talagrand.
\newblock New concentration inequalities in product spaces.
\newblock {\em Inventiones Mathematicae}, 126(3):505--563, November 1996.
\newblock \href {https://doi.org/10.1007/s002220050108}
  {\path{doi:10.1007/s002220050108}}.

\bibitem[Tro15]{Tropp15-matrix}
J.~A. Tropp.
\newblock An introduction to matrix concentration inequalities.
\newblock {\em Foundations and Trends® in Machine Learning}, 8(1-2):1--230,
  2015.
\newblock \href {https://doi.org/10.1561/2200000048}
  {\path{doi:10.1561/2200000048}}.

\bibitem[Tsy08]{Tsybakov08}
A.~B. Tsybakov.
\newblock {\em Introduction to Nonparametric Estimation}.
\newblock Springer Publishing Company, Incorporated, 2008.

\bibitem[Tuk60]{Tukey60}
J.~W. Tukey.
\newblock A survey of sampling from contaminated distributions.
\newblock {\em Contributions to probability and statistics}, 2:448--485, 1960.

\bibitem[Ver18]{Ver18}
R.~Vershynin.
\newblock {\em High-{{Dimensional Probability}}: {{An Introduction}} with
  {{Applications}} in {{Data Science}}}.
\newblock {Cambridge University Press}, 2018.

\bibitem[ZJS19]{ZHS19}
B.~Zhu, J.~Jiao, and J.~Steinhardt.
\newblock Generalized resilience and robust statistics.
\newblock {\em CoRR}, abs/1909.08755, 2019.
\newblock URL: \url{http://arxiv.org/abs/1909.08755}, \href
  {http://arxiv.org/abs/1909.08755} {\path{arXiv:1909.08755}}.

\bibitem[ZJS20]{ZhuJS2020-gradient}
B.~Zhu, J.~Jiao, and J.~Steinhardt.
\newblock Robust estimation via generalized quasi-gradients.
\newblock {\em CoRR}, abs/2005.14073, 2020.

\end{thebibliography}
\medskip

\newpage
\appendix

\section*{Appendix}

\section{Robust Mean Estimation and Stability} \label{app:stability}

\subsection{Robust Mean Estimation from Subset Stability} \label{app:subset-stab}

The theorem statement in \cite[Theorem 2.7]{DK20-survey} requires that the  input multiset $S$ is stable. 
We note that the arguments straightforwardly go through 
when $S$ contains a large stable subset $S' \subseteq S$ (see, e.g.,~\cite{DKKLMS16,DKK+17,DHL19}).

For concreteness, we describe a simple pre-processing of the data, 
that ensures that the data follows the definition as is: simply throw away points 
so that the cardinality of the corrupted set matches the cardinality of the stable subset.

\begin{proposition}
Let $S$ be a set such that $\exists S' \subseteq  S$ such that $|S'| \geq (1 - \epsilon) |S|$ and $S'$ is $(  C\epsilon,\delta)$ for some $C>0$.
Let $T$ be an $\epsilon$-corrupted version of $S$. Let $T'$ be the multiset obtained by removing  $\epsilon n$ points of $T$.
Let $\epsilon' = \frac{2 \epsilon}{1 - \epsilon}$.
Then $T'$ is an $\epsilon'$-corrupted version of a $( (C-1) \epsilon' / 2, \delta)$ stable set.
\end{proposition}
\begin{proof}

Let $T$ be an $\epsilon$-corrupted version of $S$. That is, $T = S \cup A \setminus R$.
We now remove $ \epsilon n$ points arbitrarily from $T$ to obtain the multiset $T'$ of cardinality $(1 - \epsilon)n$.

Let $S_2$ be any subset  of $S'$  such that  $|S_2| = |T_1| = (1 - \epsilon)n.$
Therefore, $T'$ is at most $ (2 \epsilon)/(1 - \epsilon)$-corrupted version of $S_2$.
As $S'$ is $(C \epsilon,\delta)$ stable and $S_2$ is a large subset of $S'$, Claim \ref{LemSubsetStability} states that $S_2$ is $( \epsilon_2, \delta)$ stable where $\epsilon_2 \geq 1 -  (1 - C\epsilon) /(1 - \epsilon) = (C - 1) \epsilon'/2$.
\end{proof}
\begin{claim}
If a set $S$ is $(\epsilon, \delta)$ stable, then its subset $S'$ of cardinality $m > (1 - \epsilon)n$ is $( 1 - (1 - \epsilon)\frac{n}{m} , \delta)$ stable.
\label{LemSubsetStability}
\end{claim}
\begin{proof}
To show that $S'$ is $(\epsilon', \delta)$ stable, it suffices to ensure that $\epsilon' \leq \epsilon$ and $(1 - \epsilon')|S'| \geq (1 - \epsilon)|S|$.
Therefore, we require that
\begin{align*}
(1 - \epsilon')m \geq (1 - \epsilon)n \implies \epsilon' \leq 1 - \frac{(1 - \epsilon)n}{m}.
\end{align*}
The upper bound is always less than $\epsilon$ for $m \leq n$.
\end{proof}

\subsection{Adapting to Unknown Upper Bound on Covariance} \label{AppUnkCov}
As stated, the stability-based algorithms in~\cite{DKK+17, DK20-survey} 
assume that the inliers are drawn from a distribution with unknown bounded covariance 
$\Sigma \preceq \sigma^2 I$, where the parameter $\sigma>0$ is known.
Here we note that essentially the same algorithms work even if the parameter $\sigma > 0$ is unknown. 
For this, we establish the following simple modification of standard results, see, e.g.,~\cite{DK20-survey}.

\begin{theorem}
Let $T \subset \R^d$ be an $\epsilon$-corrupted version of a set $S$, 
where $S$ is $(C\eps,\delta)$-stable with respect to $\mu_S$ and $\sigma^2$, where $C>0$ is a sufficiently large constant. 
There exists a polynomial time algorithm that given $T$ and $\epsilon$ (but not $\sigma$ or $\delta$) 
returns a vector $\widehat\mu$ so that $\|\mu_S-\widehat\mu\| = O(\sigma \delta)$.
\end{theorem}
\begin{proof}
The algorithm is very similar to the algorithm from \cite{DK20-survey} except for the stopping condition. 
We define a weight function $w:T\rightarrow \R_{\geq 0}$ initialized so that $w(x)=1/|T|$ for all $x\in T$. 
We iteratively do the following:
\begin{itemize}
\item Compute $\mu(w) = \frac{1}{\|w\|_1} \sum_{x\in T} w(x) x$.
\item Compute $\Sigma(w) = \frac{1}{\|w\|_1} \sum_{x\in T} w(x)(x-\mu(w))(x-\mu(w))^T$.
\item Compute an approximate largest eigenvector $v$ of $\Sigma(w)$.
\item Define $g(x)$ for $x\in T$ as $g(x) = |v\cdot( x-\mu(w))|^2$.
\item Find the largest $t$ so that $\sum_{x\in T: g(x) \geq t} w(x) \geq \eps$.
\item Define $f(x) = \begin{cases} g(x) & \textrm{if }g(x)\geq t \\ 0 & \textrm{otherwise} \end{cases}$.
\item Let $m$ be the largest value of $f(x)$ for any $x\in T$ with $w(x)\neq 0$.
\item Set $w(x)$ to $w(x)(1-f(x)/m)$ for all $x\in T$.
\end{itemize}
We then repeat this loop unless $\|w\|_1 < 1-2\eps$, in which case we return $\mu(w)$.

Note that if $S$ is $(\epsilon,\delta)$-stable with respct to $\mu_S$ and $\sigma^2$, then $S/\sigma$ is $(\epsilon,\delta)$ with respect to $\mu_S/\sigma$ and $1$. We note that if $\sigma$ was known, the weighted universal filter algorithm of \cite{DK20-survey} 
could be applied to $T/\sigma$ in order to learn $\mu_S/\sigma$ to error $O(\delta)$. Multiplying the result by $\sigma$ would yield an approximation to $\mu_S$ with error $O(\sigma \delta)$. We note that this algorithm is equivalent to the one provided above, 
except that we would stop the loop as soon as $\Sigma(w) \leq \sigma(1+O(\delta^2/\eps))$ 
rather than waiting until $\|w\|_1 \leq 1-2\eps$.

However, we note that by the analysis in \cite{DK20-survey} of this algorithm, that at each iteration until it stops, $\sum_{x\in S} w(x)$ decreases by less than $\sum_{x\in T\backslash S} w(x)$ does. Since the latter cannot decrease by more than $\eps$, this means that the algorithm of \cite{DK20-survey} would stop before ours does. Our algorithm then continues to remove an additional $O(\eps)$ mass from the weight function $w$ (but only this much since $f$ has support on points of mass only a bit more than $\eps$). It is easy to see that these extra removals do not increase $\Sigma(w)$ by more than a factor of $1+O(\eps)$. This means that when our algorithm terminates $\Sigma(w)/\sigma \leq I + O(\delta^2/\eps)$. Thus, by the weighted version of Lemma 2.4 of \cite{DK20-survey}, we have that 
$$\|\mu_S - \mu(w)\| = \sigma \|\mu_{S}/\sigma - \mu(w)/\sigma \| 
\leq \sigma O(\delta + \sqrt{\eps (\delta^2/\eps)}) = O(\sigma \delta) \;.$$
This completes the proof.
\end{proof}

\section{Tools from Concentration and Truncation}
\label{AppConcTrunc}

\paragraph{Organization.} In Section~\ref{AppConc}, we state the concentration results that we will use repeatedly in the following sections.
Section~\ref{AppConcDistprop} contains some well-known results regarding the properties of the truncated distribution.

\subsection{Concentration Results}
\label{AppConc}
We first state Talagrand's concentration inequality for bounded empirical processes.
\begin{theorem}(\cite[Theorem 12.5]{BGM-book-13} )
Let $Y_1,\dots,Y_n$ be independent identically  distributed  random  vectors.
Assume  that $\E Y_{i,s}= 0$, and  that $ Y_{i,s} \leq L$ for all $s \in \cT$.
Define
\begin{align*}
Z = \sup_{s \in \cT} \sum_{i=1}^n Y_{i,s}, \qquad \sigma^2= \sup_{s \in \cT}\sum_{i=1}^n\E Y^2_{i,s} .
\end{align*}
Then, with probability at least $1 - \exp(-t)$, we have that 
\begin{align}
Z = O( \E Z   + \sigma \sqrt{t}   +  L t).
\end{align}
See \cite[Exercise 12.15]{BGM-book-13} for explicit constants.
\label{LemTalagrandBddArbitrarily}
\end{theorem}
We will also repeatedly use the following version of Matrix Bernstein inequality~\cite{Tropp15-matrix,Min17-bern}.
\begin{theorem}(\cite[Corollary 7.3.2]{Tropp15-matrix})
Let $S_1,\dots,S_n$ be $n$ independent symmetric matrices such that $
\E S_i = 0$ and $\|S_i\|\leq L $  a.s. for each index  $i$.
Let $Z = \sum_{i=1}^n S_i$ and let $V$ be any PSD matrix such that $ \sum_{i=1}^n \E S_kS_k^T \preceq V$.
Let $\nu = \|V\|$ and $r = \srank(V)$.
Then, we have that
\begin{align}
\E \|Z\| = O(\sqrt{\nu\log r} + L \log r).
\end{align}
\label{LemMatrixConc}
In particular, if $S_i = \xi_i x_ix_i^T$, where $\xi_i$ is a Rademacher random variable, and $x_i$ is sampled independently from a distribution with  zero mean, covariance $\Sigma$, and bounded support $L$, i.e., $\|x_i\| \leq L$ almost surely.
Then $\E \|Z\| = O(\sqrt{nL \|\Sigma\|\log \srank(\Sigma)} + L \log \srank(\Sigma))$.
\end{theorem}

\subsection{Properties under Truncation}
\label{AppConcDistprop}
We state some basic results regarding truncation of a distribution in this subsection.
These results are well-known in literature and are included here for completeness
(see, e.g.,~\cite{DKK+17,LaiRV16}).

\begin{proposition}
\label{PropMeanShiftTruncBddCov}
(Shift in mean by truncation) Let $X$ be sampled from a distribution with mean $0$ and covariance $ \Sigma \preceq I $.
For a $t\geq 0$, let $g(\cdot)$ be defined as
\begin{align*}
g(x) &= \begin{cases} x, & \text{ if } x \in [-t , t],\\
t, &\text{ if } x > t, \\
-t, &\text{ if } x< -t.
\end{cases}
\end{align*}
If $t\geq C \epsilon^{-\frac{1}{2}}$, then for all $v \in \cS^{d-1}$,
$|\E g(x^Tv)| \leq  C^{-1}\sqrt{\epsilon}$.
\end{proposition}

\begin{proof}
	 Let $Z = x^Tv$. By Markov's inequality,
	\begin{align*}
\bP(Z \geq t) \leq \bP(Z^2 \geq C^2 \epsilon^{-1} ) \leq \frac{1}{C^2 \epsilon^{-1} } = C^{-2}\epsilon.
	\end{align*}

We get that \begin{align}
|\E g(Z)| = |\E Z - g(Z)| \leq \E |Z - g(Z)| \leq  \E |Z| \I _{|Z| \geq t }
		  &\leq  \sqrt{\epsilon} C^{-1}.
\end{align}
\end{proof}

\begin{proposition} (Shift in mean by truncation under higher moments) Let $X$ be sampled from a distribution with mean $0$ and covariance $(1 - \sigma_k^2 \epsilon^{1 - \frac{2}{k}})I \preceq \Sigma \preceq I $. Moreover, assume that the distribution has bounded moments, i.e., for a $k \geq 4$:
\begin{align}
\forall v \in \cS^{d-1}, \quad ( \E (v^TX)^k )^ { \frac{1}{k}} \leq \sigma_k.
\label{EqnMomentBound}
\end{align}
Note that $\sigma_2 \leq 1$. Let $T_k = \sigma_k \epsilon^{-\frac{1}{k}}$.
Then \begin{enumerate}
	\item
For all $M \in \cM$,
$
\E (x^TMx)^{\frac{k}{2}} \leq \sigma^k_k.
$

\item
For all $M \in \cM$ and $ t\geq C T_k^2$,  $
\E x^TMx\I_{x^TMx \geq t} \leq \sigma_k^2 C^{\frac{2}{k}-1}\epsilon^{1 - \frac{2}{k}}.
$
\item Let $f(\cdot)$ be defined as $f(x)= \min(x,t)$. For a $t \geq C T_k^2$,
$|\E f(x^TMx) - 1 | \leq \sigma_k^2 \epsilon^{1 - \frac{2}{k}} (1 + C^{1 - \frac{k}{2}})$.

\item Let $t \geq CT_k$. For all $v \in \cS^{d-1}$,
$|\E x^Tv\I_{|x^Tv| \leq t}| \leq \sigma_k \epsilon^{1 - \frac{1}{k}}C^{1 - k}$.

\item Let $g(\cdot)$ be defined as $g(x) = \text{sign}(x)\min(|x|,t)$. For $t \geq CT_k$ and  all $v \in \cS^{d-1}$,
$|\E g(x^Tv)| \leq \sigma_k C^{1 - k}\epsilon^{1 - \frac{1}{k}}$.

\item $
\E \|X\|^k \leq d^{\frac{k}{2}} \sigma_k^k.
$

\item
$\bP ( \|X\|  \geq  \sigma_k  \sqrt{d}\epsilon^{-1/k} )	 \leq \epsilon.$

\end{enumerate}
\label{LemmaMomentBound}

\end{proposition}

\begin{proof}
We prove each statement in turn.
\begin{enumerate}
	
\item  We use the spectral decomposition of $M$, to write $M = U^T \Delta U$, where $U$ is a rotation matrix, $\Delta$ is a non-negative diagonal matrix with diagonal entries $\lambda_i$ and trace $1$.
	Observe that if the random variable $X$ satisfies Equation~\eqref{EqnMomentBound}, then the random variable $Z:= UX$ also satisfies Equation~\eqref{EqnMomentBound}.

We use the aforementioned observation and apply Jensen's inequality to get:
	\begin{align*}
	\E (x^TMx)^{\frac{k}{2}} &= \E (Z^T \Delta Z)^{\frac{k}{2}} = \E (\sum_{i=1}^ d \lambda_i z_i^2)^{\frac{k}{2}} \leq \sum_{i=1}^d \lambda_i \E z_i^k \leq \sum_{i=1} \lambda_i \sigma_k^k \leq \sigma_k^k.
	\end{align*}
	
	\item Let $Z = x^TMx$. From the first part, we have that $\frac{k}{2}$-th moment of $Z$ is bounded by $\sigma_k^2$.
	By Markov's inequality, we get that
	 \begin{align*}
	 \bP\left\{Z \geq t \right\} \leq	 \bP\left\{Z \geq CT_k^2 \right\} \leq	 \bP\left\{Z \geq C\frac{\sigma_k^2}{\epsilon^{\frac{2}{k}}} \right\} \leq \frac{\epsilon}{C^{\frac{k}{2}}\sigma_k^k} \E Z^{\frac{k}{2}} \leq \frac{\epsilon}{C^{\frac{k}{2}}}.
	 \end{align*}
	We  can now apply H{\"o}lder's inequality to get
	\begin{align*}
	\E Z \I_{Z \geq CT_k^2} \leq \sigma_k^2 C^{ \frac{2}{k} - 1 }\epsilon^{1 - \frac{2}{k}}.
	\end{align*}

	\item As above, let $Z = x^TMx$. It follows that $ f(x) \leq x$.
	Therefore, we get that
	\begin{align*}
	\E f(x^TMx) \leq \E x^TMx \leq 1.
	\end{align*}
	For the lower bound, we get that
	\begin{align*}
	\E f(x^TMx) &\geq \E x^TMx \I_{x^TMx \leq CT_k^2} = \E x^TMx 1 - \E x^TMx \I_{x^TMx > CT_k^2} \\
	&\geq 1 - \sigma_k^2 \epsilon^{1 - \frac{2}{k}} - \sigma_k^2 \epsilon^{1 - \frac{2}{k}}C^{1 - \frac{k}{2}}.
	\end{align*}
	
	\item Let $Z = x^Tv$.
	We note that
	\begin{align*}
\bP(Z \geq t) \geq \bP(Z \geq C T_k) \leq \bP(Z^k \geq C^kT_k^k ) \leq \frac{\sigma_k^k}{\sigma_k^k \epsilon^{-1} C^k} \leq C^{-k}\epsilon.
	\end{align*}
We now bound the deviation in mean by truncation:
\begin{align*}
\E Z = \E Z\I_{|Z| \leq t} + \E Z\I_{|Z| > t} &= 0 \\
\implies |\E Z\I_{|Z| \leq t}| &= |\E Z \I_{Z >t}|\\	
			&\leq (\E Z^k )^{\frac{1}{k}} (\bP\{Z > t\})^{1 - \frac{1}{k}} \\
			&= \sigma_k C^{1 - k}\epsilon^{1 - \frac{1}{k}}.	
\end{align*}

\item Let $Z = x^Tv$.
We get that \begin{align*}
|\E g(Z)| = |\E Z - g(Z)| \leq \E |Z - g(Z)| \leq  \E |Z| \I _{|Z| \geq CT_k }
		  &\leq \sigma_k \epsilon^{1 - \frac{1}{k}}C^{1 - k}.
\end{align*}

	\item It follows by taking $M = \frac{1}{d}I$ in the first part.
	\item This follows by Markov's inequality and the previous part.
\end{enumerate}
\end{proof}

\begin{lemma}
Let $P$ be a distribution with  mean  $\mu$ and covariance $I$.
Let $X \sim P$. For $k > 2$, let its $k$-th central moment be bounded as
\begin{align*}
\text{ for all } v \in \cS^{d-1}: \quad (\E |v^TX|^k)^{\frac{1}{k}} \leq \sigma_k.
 \end{align*}
For $\epsilon\leq 0.5$, let $E$ be the event
\begin{align*}
E = \{\|X - \mu\| \leq T\},
\end{align*}
where $T$ is such that $\bP(E) \geq 1 - \epsilon$.
Let $Z$ be the random variable $X |E$, that is $X$ conditioned on $X \in E$.  Then, we have that
\begin{enumerate}
	\item $\|\mu- \E Z \| \leq \frac{1}{1 - \epsilon}\sigma_k \epsilon^{1 - \frac{1}{k}} \leq 2\sigma_k \epsilon^{1 - \frac{1}{k}}$.
	\item $ (1 - 3\sigma_k^2 \epsilon^{1 - \frac{2}{k}}) I  \preceq\cov(Z) \preceq I $.
\end{enumerate}
\label{LemDistPropAfterTrunc}
\end{lemma}
\begin{proof}
We prove each statement in turn.
\begin{enumerate}
	\item
Let $Q$ be the distribution of $Z$. We will assume that $\bP(E^c) > 0$, otherwise the results hold trivially.
Let $R$ be the distribution of $X$ conditioned on $X \in E^c$ and let $Y\sim R$.
Note that $P$ can be written as the convex combination of $Q$ and $R$.
\begin{align}
 P = (\bP(E))Q + (1 - \bP(E)) R.
 \end{align}
 Using this decomposition, we can calculate the shift in mean along any direction $v \in \cS^{d-1}$:
\begin{align*}
\bP(E) v^T \E Z + (1 - \bP E) \E v^TY &= v^T \E X = \mu\\
 \implies v^T (\E Z - \mu) &= \frac{1}{ \bP(E)} \E (-v^T(X- \mu)) \I_{X \not \in E}\\
&\leq \frac{1}{ \bP(E)}(\E |v^T(X - \mu)|^k)^{\frac{1}{k}} (\bP(E^c))^{1 - \frac{1}{k}}\\
&\leq \frac{1}{ \bP(E)}\sigma_k \epsilon^{1 - \frac{1}{k}},
\end{align*}
where the first inequality uses H\"older's inequality.
Therefore, $\|\E Z - \mu\| \leq \sigma_k \epsilon^{1 - 1/k} /(1 - \epsilon) $.
\item
We will follow the notations from the previous part.
Note that for all $v \in \cS^{d-1}$, the mean minimizes the quadratic loss
\begin{align*}
\E (v^T(Z - \E Z))^2 \leq \E (v^T(Z - \mu))^2.
\end{align*}
Note that for any direction $v$, we have that $\E (v^T(Z - \mu))^2 \leq \E (v^T(Y - \mu))^2$. As $\E(v^T(X - \mu))^2$ is the convex combination of $\E( v^TZ)^2$ and $\E (v^TY)^2$, and thus larger than the minimum of these two, we get
\begin{align*}
\E (v^T(Z- \mu))^2 &= \min(\E (v^T(Y - \mu))^2,\E (v^T(Z - \mu))^2)\\
&\leq \bP(E)\E (v^T(Z - \mu))^2 + (1 - \bP(E))\E (v^T(Y - \mu))^2 = \E (v^T(X - 	\mu))^2 = 1.
\end{align*}
Therefore, we obtain the following upper bound:
\begin{align*}
\E v^T(Z - \E Z)^2 \leq \E (v^T(Z- \mu))^2 \leq 1.
\end{align*}
We now turn our attention to lower bound. We first note that
\begin{align*}
(1 - \bP(E))\E (v^T(Y - \mu))^2 &=  \E (v^TX)^2 \I\left\{ X \in E^c \right\}\leq (\E (v^TX)^k) ^{\frac{2}{k}} ( \bP (E))^{1 - \frac{2}{k}} \leq \sigma_k^2 \epsilon^{1 - \frac{2}{k}}.
\end{align*}
Using the definition of $P$, $Q$ and $R$, we get
\begin{align*}
\E (v^T(Z - \mu))^2  &= \frac{1}{\bP(E)}( \E (v^T(X -\mu))^2 - (1 - \bP(E))\E (v^T(Y - \mu))^2) \\
					&\geq ( 1 - (1 - \bP(E))\E (v^T(Y - \mu))^2) \geq 1 - \sigma_k^2 \epsilon^{1 - \frac{2}{k}}.
\end{align*}
We are now ready to bound the lower bound the deviation from mean:
\begin{align*}
\E (v^T(Z - \E Z))^2 &= \E (v^T(Z - \mu))^2 -  (\E Z - \mu)^2 \\
   &\geq  1 - \sigma_k^2 \epsilon^{1 - \frac{2}{k}} -  (\frac{ \sigma_k \epsilon^{1 - \frac{1}{k}}}{1 - \epsilon})^2 \\
			&\geq 1 - \sigma_k^2 \epsilon^{1 - \frac{2}{k}} - \frac{\sigma_k^2 \epsilon^{1 - \frac{2}{k}}}{1 - \epsilon} \geq 1 - 3 \sigma_k^2 \epsilon^{1 - \frac{2}{k}}.
\end{align*}
\end{enumerate}

\end{proof}

\section{Bounds on the Number of Points with Large Projections} %
\label{AppProjOutlier}

\paragraph{Organization.} This section contains the proofs of Lemma~\ref{LemBddMatrixProj} and Lemma~\ref{LemBddMatrixProjMom} from the main paper.
In Section~\ref{AppLinProj}, we prove the results controlling the number of outliers uniformly along all directions $v\in \cS^{d-1}$. We then generalize these results to projections along PSD matrices in Section~\ref{AppMatProj}.

\subsection{Linear Projections} \label{AppLinProj}
We state Lemma 1 from Lugosi and Mendelson~\cite{LugosiM19robust}.
We will use this result for distributions with bounded covariance.
\begin{lemma}(\cite[Lemma 1]{LugosiM19robust})
Let $x_1,\dots,x_n$ be $n$ i.i.d. points from a distribution with mean zero and covariance $\Sigma \preceq I$.
Let $Q_2$ be defined as follows:
\begin{align*}
Q_2 =  \frac{256}{\epsilon} \sqrt{\frac{\trace(\Sigma)}{n}} + \frac{16}{\sqrt{\epsilon}}.
\end{align*}
Then, for a constant $c > 0$,  with probability at least $1 - \exp( - c\epsilon n)$,
\begin{align*}
 \sup_{v \in \cS^{d-1}}  \left|\left\{ i: |v^Tx_i| \geq Q_2  \right\} \right|  \leq  0.25\epsilon n.
\end{align*}
\label{LemBddVccProj}
\end{lemma}

We state the following straightforward generalization of Lemma~\ref{LemBddVccProj} for distributions with bounded central moments.
We give the proof for completeness.
\begin{lemma}
Let $x_1,\dots,x_n$ be $n$ i.i.d. points from a distribution with mean zero and covariance $\Sigma \preceq I$.
Further assume that for all $v \in \cS^{d-1}$:
\begin{align}
 (\E (v^TX)^k)^{\frac{1}{k}} \leq \sigma_k.
  \end{align}
Let $Q_k$ be defined as follows:
\begin{align*}
Q_k =  \Theta\left(\frac{1}{\epsilon} \sqrt{\frac{\trace(\Sigma)}{n}} + \sigma_k \epsilon^{- \frac{1}{k}} \right).
\end{align*}
Then, there exists a $c > 0$, such that with probability at least $1 - \exp( - c n \epsilon)$,
\begin{align}
 \sup_{v \in \cS^{d-1}}  \left|\left\{ i: |x_i^Tv| \geq Q_{k}   \right\} \right|  = O(   n \epsilon).
\end{align}
\label{LemBddVecProjMoments}
\end{lemma}
\begin{proof}
We follow the same strategy as in Lugosi and Mendelson~\cite{LugosiM19robust}.
We first set $Q_k$ as follows:
\begin{align*}
Q_k =  C\left(\frac{1}{\epsilon} \sqrt{\frac{\trace(\Sigma)}{n}} + \sigma_k \epsilon^{- \frac{1}{k}} \right),
\end{align*}
for a large enough constant $C$ to be determined later.
Consider the  function $\chi: \R \to \R$ defined by
\begin{align}
\chi(x) = \begin{cases}0, & \text{if } x \leq \frac{Q_{k}}{2},\\
\frac{2x}{Q_{k}} - 1, & \text{if } x \in \left[\frac{Q_{k}}{2}, Q_{k}\right],\\
1, & \text{if } x \geq Q_{k}.\end{cases}
\end{align}
Therefore, $\I_{x^Tv \geq Q_{k}} \leq \chi(x_i^Tv) \leq \I_{x^Tv \geq Q_{k}/2}$ and note that $\chi(\cdot)$ is a $\frac{2}{Q_{k}}$ Lipschitz.
We first bound the number of points violating the upper tail bounds.
The random quantity of interest is the following:
\begin{align}
Z = \sup_{v \in \cS^{d-1}}\sum_{i=1}^n \I_{ x_i^Tv \geq Q_{k} } \;.
\end{align}
We first calculate its expectation using the symmetrization principle~\cite{LedouxTalagrand,BGM-book-13}. We have that
\begin{align}
\nonumber \E Z &=
 \E \sup_{v \in \cS^{d-1}}\sum_{i=1}^n \I_{ x_i^Tv \geq Q_{k} }\\
 	\nonumber&\leq  \E \sup_{v \in \cS^{d-1}}\sum_{i=1}^n \chi( x_i^Tv)\\
	\nonumber&\leq  \E \sup_{v \in \cS^{d-1}}\sum_{i=1}^n ( \chi( x_i^Tv)  - \E   \chi( x_i^Tv)) + \sup_{v \in \cS^{d-1}} \E \sum_{i=1}^n \chi( x_i^Tv) \\
	\label{EqnProjMoments}&\leq 2\E \sup_{v \in \cS^{d-1}} \sum_{i=1}^n \epsilon_i \chi( x_i^Tv)  + \sup_{v \in \cS^{d-1}} \E \sum_{i=1}^n \chi( x_i^Tv).
\end{align}
We bound the second term in Eq.~\eqref{EqnProjMoments} by
\begin{align*}
\E \sum_{i=1}^n \chi( x_i^Tv) \leq \E \sum_{i=1}^n \I_{ x_i^Tv|\ge Q_{k}/2 } = n \bP( x_i^Tv \geq Q_{k}/2 ) \leq n \bP( x_i^Tv \geq C \sigma_k \epsilon^{- \frac{1}{k}} )   = O(n \epsilon),
\end{align*}
by applying Markov inequality and choosing a large enough constant $C$ for $Q_k$.
For the first term in Eq.~\eqref{EqnProjMoments}, we upper bound $\chi(\cdot)$ using contraction principle for Rademacher averages and independence of $x_i$:
\begin{align*}
\E \sup_{v \in \cS^{d-1}} \sum_{i=1}^n \epsilon_i \chi( x_i^Tv) &\leq  \frac{2}{Q_{k}} \E \sup_{v \in \cS^{d-1}} \sum_{i=1}^n \epsilon_i   x_i^Tv = \frac{2}{Q_{k}} \E \|\sum_i \epsilon_i x_i\|\leq n\frac{2}{Q_{k}} \sqrt{n \trace(\Sigma)} = O(n \epsilon),	
\end{align*}
where we use the covariance bound on $x_i$ and a large enough constant for $Q_k \geq (C / \epsilon)\sqrt{\trace(\Sigma) /n}$.
Therefore, we get that $\E Z = O( n \epsilon)$.
We can bound the wimpy variance, i.e., the quantity $\sigma^2$ in Theorem~\ref{LemTalagrandBddArbitrarily}, by $O(\epsilon n)$.
By Talagrand's concentration~\ref{LemTalagrandBddArbitrarily}, we get that probability $1 - \exp(- c n \epsilon)$,
\begin{align}
 Z = O( n \epsilon + \sqrt{ n \sigma  } \sqrt{c n \epsilon}  \sqrt{ n \gamma} +  c n \epsilon) = O(n \epsilon).
 \end{align}
 \end{proof}

\subsection{Matrix Projections}\label{AppMatProj}

We will now use the results from the previous section to prove  Lemma~\ref{LemBddMatrixProj} and Lemma~\ref{LemBddMatrixProjMom}.
The proof follows the ideas from \cite[Proposition 1]{DepLec19}.
\begin{lemma}
Suppose that the event $\mathcal{E}_1$ holds, where $\mathcal{E}_1$ is the following
\begin{align*}
\mathcal{E}_1 := \left\{\sup_{v \in \cS^{d-1}}|\{i: |x_i^Tv| \geq Q_0 \}| \leq 0.25\epsilon n \right\}.
\end{align*}
Let $Q = 8 Q_0$ and $\epsilon\geq{1/n}$.
Then the event $\mathcal{E}$ also holds, where $\mathcal{E}$ is defined as follows:
\begin{align*}
\mathcal{E} := \left\{ \sup_{M \in \cM} |\{i : x_i^TMx_i \geq Q^2\}| \leq \epsilon  n\right\}.
\end{align*}
\label{LemBddMatrixProj2}
\end{lemma}
\begin{proof}
We follow the same proof strategy as Depersin and Lecué~\cite{DepLec19}.
We reproduce the proof here for completeness.

Suppose that $\mathcal{E}_1$ holds but the desired event $\mathcal E$ does not hold.
Let $M$ be such that $|\{i : x_i^TMx_i \geq  Q^2\}| > \epsilon n$.
Let $G$ be the Gaussian vector in $\R^d$ independent of $x_1,\dots,x_n$ with distribution $\mathcal{N}(0,M)$.
We will work conditionally on $x_1,\dots,x_n$ in the remaining of the proof.
Let $Z$ be the following random variable
\begin{align*}
Z = \sum_{i=1}^n \I_{ |x_i^TG|^2 \geq 25 Q^2_0}\;.
\end{align*}
We have that $x_i^TG \sim  \mathcal{N}(0, x_i^TMx_i)$.
For $i$ such that  $x_i^TMx_i \geq Q^2$, we have that
\begin{align*}
\bP(|x_i^TG|^2 > 25Q^2_0  ) \geq2\bP(g \geq \frac{5}{8}) > 0.528,
\end{align*}
where $g$ is a standard Gaussian random variable. Therefore,
\begin{align*}
\E Z = \sum_{i=1}^n \bP (|x_i^TG|^2> 25 Q_0^2) \geq \epsilon n (0.528) \geq 0.528.
\end{align*}
Note that $Z$ a is sum of independent indicator random variables.
A Chernoff bound (see, e.g., \cite[Section 2.3]{Ver18}) states that, with probability at least $1 -  (\sqrt{2/e})^{\E Z }  > 0.05$, we have that 
$Z \geq \frac{\E Z}{2} > 0.25 n \epsilon$.
However, by Gaussian concentration (see, e.g., \cite{BGM-book-13}) we have that with probability at least $0.9999$:
$\|G\| \leq 5$.
Taking a union bound, we get that both of the events happen simultaneously with non-zero probability.
Therefore, with non-zero probability $\exists u : \|u\| \leq 5$ and
\begin{align*}
\sum_{i=1}^n \I_{|x_i^Tu|^2 \geq 25Q_0^2} > 0.25n \epsilon.
\end{align*}
That is, $\exists v: \|v\| \leq 1$, and
\begin{align*}
\sum_{i=1}^n \I_{|x_i^Tv|^2 \geq Q_0^2} > 0.25n \epsilon \quad\equiv \quad \sum_{i=1}^n \I_{|x_i^Tv| \geq Q_0} > 0.25n \epsilon,
\end{align*}

which is a contradiction to $\mathcal E_1$. This completes the proof.
\end{proof}

We are now ready to prove Lemma~\ref{LemBddMatrixProj} and \ref{LemBddMatrixProjMom}.

\begin{proof}(Proof of Lemma~\ref{LemBddMatrixProj}) Without loss of generality, we can assume $\epsilon n = \Omega(1)$. The result now follows from Lemma~\ref{LemBddVccProj}, due to Lugosi and Mendelson~\cite[Lemma 1]{LugosiM19robust}, and Lemma~\ref{LemBddMatrixProj2}.
\end{proof}
\begin{proof}(Proof of Lemma~\ref{LemBddMatrixProjMom}) Without loss of generality, we can assume $\epsilon n = \Omega(1)$. The result now follows from Lemma~\ref{LemBddVecProjMoments}, which might require a change of variables, and Lemma~\ref{LemBddMatrixProj2}.
\end{proof}

\section{Stability for Distributions with Bounded Covariance} \label{AppBddCov}
\paragraph{Organization.}  Section~\ref{AppCovSuff} contains the proof of the sufficient conditions for stability under bounded covariance assumption (Claim~\ref{ClaimCovSuffStabMain}). Section~\ref{AppCovDetRounding} contains the arguments for  deterministic rounding (Lemma~\ref{LemDeterministicRounding}).

\subsection{Sufficient Conditions for Stability} \label{AppCovSuff}

The following claim simplifies the stability condition for the bounded covariance case.
\begin{claim}(Claim~\ref{ClaimCovSuffStabMain})
Let $S$ be a set such that $\|\mu_S - \mu\| \leq \sigma\delta$, and $\|\overline{\Sigma}_S - \sigma^2I\| \leq \sigma^2\delta^2 /\epsilon$ for some $ 0 \leq \epsilon \leq  \delta $.
Let $\epsilon' < 0.5$.
Then $S$ is $(\epsilon', \delta')$ stable with respect to $\mu$ and $\sigma^2$, where $\delta' = 2\sqrt{\epsilon'} + 2 \delta \sqrt{\epsilon'/\epsilon}$.
\label{ClaimCovSuffStab}
\end{claim}
\begin{proof}
Let $\epsilon' < 0.5$. Without loss of generality, we can assume that $\sigma = 1$.
For $S'\subseteq S: |S'| \geq (1 - \epsilon')|S|$,
\begin{align*}
\frac{1}{|S'|} \sum_{i \in S'} (x_i^Tv)^2  -1 &\leq \frac{1}{|S'|} \sum_{i \in S} (x_i^Tv)^2 -1 \leq \frac{1}{1 - \epsilon'}(1 + \frac{\delta^2}{\epsilon}) -1  \\
&= \frac{ \frac{\delta^2}{\epsilon} + \epsilon'}{1 - \epsilon'} \leq \frac{1}{\epsilon'}(2 \epsilon' + 2\delta \sqrt{\frac{\epsilon'}{\epsilon}} )^2 \leq \frac{(\delta')^2}{\epsilon'}.
\end{align*}
As $\delta' \geq \sqrt{\epsilon'}$, the lower bound on eigenvalues of $\overline{\Sigma}_{S'}$ is trivially satisfied.
We now bound the deviation in mean.
Observe that the uniform distribution on $S'$ can be obtained by conditioning the uniform distribution on $S$ on an event $E$, such that $\bP(E) \geq 1 - \epsilon'$.
Using this observation in conjunction with H\"older's inequality gives us that for any $v$, the shift in mean is at most
\begin{align}
\left|\frac{1}{|S'|} \sum_{i \in S'} v^Tx_i  - \frac{1}{|S|} \sum_{i \in S'} v^Tx_i \right| \leq  2\sqrt{1 + \frac{\delta^2}{\epsilon}} \sqrt{\epsilon'} \leq 2 \sqrt{\epsilon'} + 2 \delta \sqrt{\frac{\epsilon'}{\epsilon}} \leq \delta '.\end{align}

\end{proof}

\subsection{Deterministic Rounding of the Weight Function}
\label{AppCovDetRounding}

The next lemma states that it suffices to find a distribution $w \in \Delta_{n ,\epsilon}$ for stability.
\begin{lemma}(Lemma~\ref{LemDeterministicRoundingMain})
For $\epsilon \leq \frac{1}{3}$, let $w^* \in \Delta_{n,\epsilon}$ be such that for $\epsilon \leq \delta$, we have
\begin{enumerate}
	\item $\|\mu_w - \mu\| \leq \sigma\delta$.
	\item $\| \overline{\Sigma}_w - \sigma^2 I\| \leq \sigma^2\delta^2 / \epsilon$.
\end{enumerate}
Then there exists a subset $S_1 \subseteq S$ such that
\begin{enumerate}
	\item $|S_1| \geq (1 - 2\epsilon)|S|$.
	\item $S_1$ is $(\epsilon',  \delta')$ stable with respect to $\mu$ and $\sigma^2$, where $\delta' = O( \delta + \sqrt{\epsilon} + \sqrt{\epsilon'}  )$.	
	
	\end{enumerate}
\label{LemDeterministicRounding}
\end{lemma}
\begin{proof}
Without loss of generality, we will assume that $\sigma^2 = 1$.
We will use  Claim~\ref{ClaimCovSuffStab} to prove this result by first showing that there exists a $S'\subseteq [n]$ with bounded covariance and good sample mean.

Without loss of generality, we will assume that $\epsilon n$ is an integer and $\mu= 0$.
We will also assume that $\frac{1}{(1 - \epsilon)n}\geq w_1\geq w_2\geq \dots \geq w_n \geq 0$.
For any $k \in [n]$, we have that
\begin{align}
1 &= \sum_i w_i \leq \frac{n - k}{(1 - \epsilon)n} + k w_k \\
\implies w_k &\geq \frac{1}{k}\frac{(1 - \epsilon)n - (n-k)}{(1- \epsilon)n} = \frac{k - \epsilon n}{(1 - \epsilon)nk}.
\end{align}
Setting $k =  2\epsilon n$, we have that
\begin{align}
w_{k} \geq \frac{2 \epsilon n}{2 n (1 - \epsilon)} = \frac{1}{2(1 - \epsilon)n}.
\label{EqDetRoundingLow}
\end{align}
We now have a lower bound on $w_i$ for all $i \leq (1 - 2 \epsilon) n$.
Now let $S_1$ be the set of the $n-k$ points with the largest $w_i$. In particular, for each $i \in S_1$, $w_i \geq \frac{1}{2(1 - \epsilon )n}$.
We have that,
\begin{align}
\nonumber \sum_{i \in S_1} \frac{1}{|S_1|} (x_i^Tv)^2 &=    \sum_{i \in S_1} \frac{1}{(1 - 2 \epsilon)n} (x_i^Tv)^2\\
 &\leq \sum_{i \in S_1} \frac{1}{(1 - 2 \epsilon)} 2w_i(1 - \epsilon)(x_i^Tv)^2
 \nonumber \tag{Using Eq.~\eqref{EqDetRoundingLow} }\\
 &\leq \frac{2 (1 - \epsilon)}{(1 - 2 \epsilon)}\sum_{i \in S} w_i (x_i^Tv)^2
 \nonumber \\
 &\leq 9 (1 + \frac{\delta^2}{\epsilon}).
 \label{EqnVarDetRounding}
 \end{align}
 Let the uniform distribution on $S_1$ be $u^{(1)}$ and the uniform distribution on $S$ be $u$.
We now calculate the total variation distance between $w$ and $u^{(1)}$.
\begin{align}
d_{\text{TV}} (w,u^{(1)}) \leq d_{\text{TV}} (w,u) + d_{\text{TV}} (u,u^{(1)}) \leq \epsilon + 2 \epsilon = 3 \epsilon.
\end{align}
Therefore, there exist distributions $p^{(1)},p^{(2)},p^{(3)}$ such that
\begin{align*}
w = (1 - 3\epsilon)p^{(1)} + 3 \epsilon p^{(2)}, \qquad u_1 = (1 - 3 \epsilon)p^{(1)} + 3 \epsilon p^{(3)}.
\end{align*}
This decomposition follows from an alternate characterization of total variation distance(see, e.g., \cite[Lemma 2.1]{Tsybakov08}).
We first note that
\begin{align*}
 3 \epsilon \sum_i p^{(2)}_i (x_i^Tv)^2  \leq  \sum_i w_i (x_i^Tv)^2 \leq 1 + \frac{\delta^2}{\epsilon}.
 3 \epsilon \sum_i p^{(3)}_i (x_i^Tv)^2  \leq  \sum_i u^{(1)}_i (x_i^Tv)^2 \leq 9\left(1 + \frac{\delta^2}{\epsilon}\right).
\end{align*}
Therefore, we get that
\begin{align*}
|\sum_{i=1}^n (1 - 3\epsilon)p^{(1)}_i x_i^Tv | &\leq  |\sum_{i=1}^n w_i x_i^Tv | + |3 \epsilon \sum_i p^{(3)}_i x_i^Tv| \leq \delta + 3 \epsilon \sqrt{\sum_{i=1}^n p_i (x_i^v)^2} \\
& \leq \delta + \sqrt{3 \epsilon} \sqrt{ 3 \epsilon \sum_{i=1}^n p_i (x_i^Tv)^2} \leq \delta + \sqrt{3 \epsilon} \sqrt{(1 + \frac{\delta^2}{\epsilon}}) \\
&\leq \delta + \sqrt{3 \epsilon} + \sqrt{3}\delta \leq 3 \delta + 2 \sqrt{\epsilon}.
\end{align*}
We finally get that
\begin{align}
|\sum_{i=1}^n u^{(1)}_i x_i^Tv | &\leq |\sum_{i=1}^n (1 - 3\epsilon)p^{(1)}_i x_i^Tv| + |\sum_{i=1}^n 3\epsilon p^{(3)}_i x_i^Tv | \nonumber\\
&\leq 3 \delta + 2 \sqrt{\epsilon} + \sqrt{3 \epsilon} \sqrt{ 3 \epsilon \sum_i p^{(3)}_i (x_i^Tv)^2} \nonumber\\
&\leq 3 \delta + 2 \sqrt{\epsilon} + \sqrt{ 27} \sqrt{ \epsilon + \delta^2} \leq 10 \delta + 10 \sqrt{\epsilon}.
\label{EqnMeanRounding}
\end{align}

Therefore using Equations~\eqref{EqnVarDetRounding} and \eqref{EqnMeanRounding}, we have a set $S_1$ that satisfies the conditions in Claim~\ref{ClaimCovSuffStab} with $\delta'' = 10 \delta + 10 \sqrt{\epsilon}$.
Using Claim~\ref{ClaimCovSuffStab}, we get that $S_1$ is $(\epsilon' , \delta')$ stable.
\end{proof}

\section{Stability for Distributions with Bounded Central Moments}
\label{AppHighMoments}
\paragraph{Organization.}
In this section, we provide the detailed arguments regarding the proof of Theorem~\ref{ThmStabilityHighMom} that were omitted from the main text.
We start with a simplified stability condition in Section~\ref{AppHighSuff}.
Section~\ref{AppHighRandRound} contains the argument for rounding a good  distribution $w\in \Delta_{n ,\epsilon}$ to a subset.
Section~\ref{AppHighVarUpp} contains the arguments for controlling the second moment matrix from above and below respectively.
Sections~\ref{AppHighVarUpp} and~\ref{AppHighMean} contain the arguments for concentration of the second moment matrix and mean respectively.

\subsection{Sufficient Conditions for Stability}
\label{AppHighSuff}

We will prove the existence of a stable set with high probability using the following claim. This is analogous to Claim~\ref{ClaimCovSuffStab} in the bounded covariance setting, but we also need a lower bound on the minimum eigenvalue of $\overline{\Sigma}_{S'}$ for all large subsets $S'$.
\begin{claim} Let $0 \leq \epsilon \leq \delta$ and  $\epsilon \leq 0.5$.  A set $S$ is $(\epsilon, 7\delta)$ stable, if it satisfies the following for all unit vectors $v$.
\begin{enumerate}
	\item $\|\mu_{S} - \mu\| \leq \delta$.
	\item $ v^T \overline{\Sigma}_S v \leq 1 + \frac{\delta^2}{\epsilon}$.
	\item For all subsets $S' \subseteq S$ such that $|S'| \geq (1 -\epsilon) |S|$, we have $ v^T \overline{\Sigma}_{S'} v \geq (1 - \frac{\delta^2}{\epsilon})$.
\end{enumerate}
\label{ClaimSuffStabHighMom}
\end{claim}
\begin{proof}
Without loss of generality, we will assume that $\mu= 0$.
We first show the second condition in the definition of stability. Let $S'$ be any proper subset of $S$, such that $|S'| \geq (1 - \epsilon) |S|$.
Note that the minimum eigenvalue of $S'$ is lower-bounded by the assumption:
\begin{align}
v^T \Sigma_{S'} v = \frac{1}{|S \setminus S_\epsilon|}\sum_{i \in S \setminus S_\epsilon} (v^Tx)^2 \geq 1 - \frac{\delta^2}{\epsilon}.
\end{align}
We now look at the largest eigenvalue of $S'$:
\begin{align*}
v^T \Sigma_{S}v -1 &= \frac{1}{|S'|}\sum_{i \in S'} (v^Tx)^2 - 1 \leq \frac{|S|}{|S'|}\frac{1}{|S|}\sum_{i \in S} (v^Tx)^2 - 1 \\
&\leq \frac{1}{1 - \epsilon }(1 + \frac{\delta^2}{\epsilon}) - 1\leq \frac{1}{1 - \epsilon}(\frac{\delta^2}{\epsilon} + \epsilon) \leq \frac{2 \delta^2}{\epsilon} + 2 \epsilon \leq 4\frac{\delta^2}{\epsilon}.
\end{align*}
We now need to show that the mean of $S'$ is also good. In order to do that, we first control the deviation due to a small set $S \setminus S'$.
\begin{align}
\nonumber
\frac{1}{|S|}\sum_{i \in S \setminus S'} (v^Tx_i)^2 &= \frac{1}{|S|} \sum_{i \in S} (v^Tx_i)^2  - \frac{1}{|S|}(\sum_{i \in S'} (v^Tx_i)^2  ) \\
					&\leq (1 + \frac{\delta^2}{\epsilon}) - \frac{|S '|}{|S|}(1 - \frac{\delta^2}{\epsilon})\nonumber\\
					&\leq (1 + \frac{\delta^2}{\epsilon}) - (1 - \epsilon)(1 - \frac{\delta^2}{\epsilon})\leq \frac{2\delta^2}{\epsilon}  + \epsilon \label{EqChangeInSecondMoment}.
\end{align}
We first break the deviation in mean into two terms, and control each individually:
\begin{align*}
\left|\frac{1}{|S'|}\sum_{i \in S'} (v^T x_i)\right| &= \frac{|S|}{|S'|} \left|\frac{1}{|S|}\sum_{i \in S \setminus S_\epsilon} (v^T x_i)\right| \leq \frac{|S|}{|S'|} \left|\frac{1}{|S|}\sum_{i \in S } (v^T x_i)\right| + \frac{|S|}{|S '|} \left|\frac{1}{|S|}\sum_{i \in S \setminus S' } (v^T x_i)\right|.
\end{align*}
We can upper bound the first term by $ \|\mu_S\|/(1 - \epsilon)  \leq \delta/(1 - \epsilon)$. We bound the second term using the Cauchy-Schwarz inequality and Eq.~\eqref{EqChangeInSecondMoment}:
\begin{align*}
\frac{|S|}{|S '|} \left|\frac{1}{|S|}\sum_{i \in S \setminus S' } (v^T x_i)\right|		&\leq \frac{|S \setminus S'|}{|S'|} \left|\frac{1}{|S \setminus S'|}\sum_{i \in S \setminus S' } (v^T x_i)\right| \\
&\leq  \frac{|S \setminus S'|}{|S'|} \sqrt{\frac{1}{|S \setminus S'|}\sum_{i \in S \setminus S' } (v^T x_i)^2 }\\
		&= \frac{\sqrt{|S \setminus S'| |S|}}{|S'|} \sqrt{\frac{1}{|S|}\sum_{i \in S \setminus S' } (v^T x_i)^2 }
		\leq \frac{\sqrt{\epsilon}}{1 - \epsilon} \sqrt{ \frac{2 \delta^2 }{\epsilon} + \epsilon }.
\end{align*}
Overall, we get that
\begin{align*}
|v^T \mu_{S'}| \leq \frac{1}{1- \epsilon}(\delta + \sqrt{2}\delta + \epsilon) \leq 5 \delta+ 2 \epsilon \leq 7 \delta.
\end{align*}
\end{proof}

\subsection{Randomized Rounding of Weight Function}
\label{AppHighRandRound}
In this section, we show how to recover a subset from a $w \in \Delta_{n , \epsilon}$.
Unlike the deterministic rounding in Section~\ref{AppCovDetRounding}, we do a randomized rounding in Lemma~\ref{LemRandRound} to get a better dependence on $\epsilon$.
For the second condition ($\delta^2 = O(\epsilon)$) in Lemma~\ref{LemRandRound} to hold, it is necessary that $n = \Omega(d)$. If $n = O(d)$, it is not a problem because, in this regime, the bounded covariance assumption already leads to optimal error.
\begin{lemma}
Let $k \geq 4$.
Let $w \in \Delta_{n,\epsilon}$, for $\epsilon \leq \frac{1}{3}$, be a distribution on the set of points $S$ such that
\begin{enumerate}
	\item $\|\mu_w - \mu\| \leq \delta$.
	\item $\|\overline{\Sigma}_w\| - 1 \leq \frac{\delta^2}{\epsilon} \leq r_1$, for some $r_1 > 1$.
	\item  Let $C \geq 4$. For all subsets $S'$: $|S'| \geq (1 - C\epsilon)n$ and $v \in \cS^{d-1}$: $v^T \overline{\Sigma}_{S'} v \geq 1 - \delta^2/(C \epsilon)$.
	\item $w_i > 0$ implies that $ \|x_i\| \leq r_2\sigma_k  \sqrt{d}\gamma^{-1/k} $ for some $r_2 \geq 1$.
\end{enumerate}
Then, there exists a subset $S_1 \subseteq [n]$ such that
\begin{enumerate}
	\item $|S_1| \geq (1 - 2\epsilon)n$.
	\item $S_1$ is $(\epsilon',  \delta')$ stable, where
	\begin{align}
	\epsilon' = (C - 2) \epsilon	, \qquad \delta' = O\Big( \delta + \sqrt{ \frac{r_1d \log d}{n }} + r_2\sigma_k \epsilon^{\frac{1}{2} - \frac{1}{k}} \sqrt{\frac{d\log d}{n}}+ r_2r_1\sigma_k \epsilon^{1 - \frac{1}{k}}\Big).	
	\end{align}
	
	\end{enumerate}
\label{LemRandRound}
\end{lemma}
\begin{proof}
We will use Claim~\ref{ClaimSuffStabHighMom} to prove this result.
Without loss of generality, let $\mu = 0$.
Therefore, it suffices to find a subset such that both the mean and the largest eigenvalue are controlled.
Let $Y_i\sim \text{Bernoulli}(w_i (1 - \epsilon)n)$. We have that $ \sum_{i=1}^n \E Y_i = (1 - \epsilon ) n$.
Let $S_1$ be the (random) set:
\begin{align}
S_1 = \{ i: Y_i = 1 \}.
\end{align}
By a Chernoff bound, we have that for some constant $c'>0$,
\begin{align}
\bP(|S_1| \geq (1 - 2 \epsilon) n) \leq \exp(-c' n \epsilon).
\label{EqHMRoundCardinality}
\end{align}
Let $E$ be the event $E = \{ |S_1| \geq (1 - 2 \epsilon) n\}$.
We now bound the mean of the set $S_1$. Consider the following random variable $Z$:
\begin{align}
Z = \sum_{i}(Y_i - (1- \epsilon)w_in)x_i.
\end{align}
The random variable $Z$ satisfies $\E Z = 0$. Moreover, its covariance can be bounded using the assumption as follows:
\begin{align*}
v^T\Sigma_Zv &= \sum_{i=1}^n w_i(1- \epsilon)n (1 - w_i(1- \epsilon)n) (v^Tx_i)^2\\
			&\leq  (1- \epsilon)n \sum_{i=1}^n w_i (x_i^Tv)^2 \leq (1 - \epsilon)n (1 + \frac{\delta^2}{\epsilon}) \preceq 2r_1n.
\end{align*}
Therefore, with  probability at least $0.8$, we have that
\begin{align*}
\|Z\| &\leq 10\sqrt{ r_1 n d }\\
\implies \|\sum Y_i x_i\| &\leq (1 - \epsilon) n \|\sum_{i} w_iX_i\| + 10 \sqrt{ r_1 n d }.
\end{align*}
Let $E_2$ be the event that $ E_2 = \{\|\sum Y_i x_i\| \leq (1 - \epsilon) n \sigma + 10 \sqrt{ r_1 n d }\} $.
This implies that on the event $E \cap E_1$,
\begin{align}
\left\|\mu_{S_1}\right\| \leq \frac{1 - \epsilon}{1 - 2 \epsilon}\delta + 10\frac{c_5}{1- 2 \epsilon} \sqrt{\frac{d}{n}} \leq 2 \delta + 30 \sqrt{\frac{r_1d}{n}}.	
\label{EqHMRoundMean}
\end{align}
We now focus our attention on upper bounding the eigenvalue. Define the symmetric random matrix, $Z_i$ as $Z_i := Y_ix_ix_i^T - w_i(1 - \epsilon)n x_ix_i^T$.
 We have that $\E Z_i = 0$ and $\|Z_i\| \leq r_2^2d \sigma_k \epsilon^{1 - \frac{1}{k}}$ almost surely.
 We now bound the matrix variance statistic (used in Theorem~\ref{LemMatrixConc}):
 \begin{align*}
 \nu(Z) &=\left\| \sum_{i=1}^n w_i (1 - \epsilon)n(1 - w_i (1 - \epsilon)n) \|x_i\|^2 x_ix_i^T  \right\| \\
 &\leq \left\| \sum_{i=1}^n w_i (1 - \epsilon)n \frac{r_2^2\sigma_k^2  d}{\epsilon^{\frac{2}{k}}} x_ix_i^T  \right\| \\
 &\leq  (1 -\epsilon)\frac{r_2^2\sigma_k^2  nd}{\epsilon^{\frac{2}{k}}} \left\| \sum_{i=1}^n w_ix_ix_i^T \right\| \\
 &\leq (1 -\epsilon)\frac{r_2^2\sigma_k^2  nd}{\epsilon^{\frac{2}{k}}} \|\overline{\Sigma}_w\| \leq 2\frac{r_1r_2^2\sigma_k^2  nd}{\epsilon^{\frac{2}{k}}}.
 \end{align*}
  By the matrix concentration (Theorem~\ref{LemMatrixConc}), we get that with probability at least $0.8$, we have that
 \begin{align}
\left \| \sum_{i=1}^n Y_ix_ix_i^T - w_i(1 - \epsilon)n x_ix_i^T \right\| = O\left (\sqrt{ \frac{r_1r_2^2 \sigma_k^2  nd \log d}{\epsilon^{\frac{2}{k}}} } + \frac{ r_2^2\sigma_k^2 d \log d}{\epsilon^{\frac{2}{k}}} \right)	.
 \end{align}
 Let $E_3$ be the event above, which happens with probability at least $0.8$.
 Under the event $E \cap E_3$, we get that
 \begin{align}
\nonumber v^T \overline{\Sigma}_{S_1} v &\leq \frac{1 - \epsilon}{1 - 2 \epsilon} w_i (x_i^Tv)^2 + \frac{1}{1 - 2 \epsilon}O\left(\sqrt{ \frac{ r_1r_2^2\sigma_k^2 d \log d}{n \epsilon^{\frac{2}{k}}} } + \frac{r_2^2 \sigma_k^2 d \log d}{n \epsilon^{\frac{2}{k}}} \right)\\
		\nonumber &\leq \frac{1 - \epsilon}{1 - 2 \epsilon}(1 + \frac{\delta^2}{\epsilon}) +  O\left(\sqrt{ \frac{ r_1r_2^2\sigma_k^2 d \log d}{n \epsilon^{\frac{2}{k}}} } + \frac{r_2^2 \sigma_k^2 d \log d}{n \epsilon^{\frac{2}{k}}} \right) \\
\nonumber		&\leq 1 + \frac{1}{\epsilon} O\left(\epsilon^2 + \delta^2 + \sqrt{ \frac{  d \log d}{n } } r_1 r_2 \sigma_k\epsilon^{1 - \frac{1}{k}} + r_2^2 \sigma_k^2 \epsilon^{1 - \frac{2}{k}} \frac{ d \log d}{n } \right) \\
		&\leq 1 + \frac{1}{\epsilon} \left(O\left( \delta +  r_1r_2\sigma_k \epsilon ^{1 - \frac{1}{k}} + \sqrt{\frac{d\log d}{n}}  + r_2 \sigma_k \epsilon^{\frac{1}{2} - \frac{1}{k}}\sqrt{\frac{d \log d}{n}}\right) \right)^2 .
\label{EqHMRoundVar} \end{align}

Let $\epsilon' =  (C-2)\epsilon $.
Note that if $|S_1| \geq (1 - 2 \epsilon)|S|$, then $ |S'| \geq (1 - \epsilon')|S_1|$ implies that $|S'| \geq (1 - C \epsilon) |S|$, which leads to a lower bound on the minimum eigenvalue. This follows from the following elementary calculations:
\begin{align}
\frac{|S'|}{|S|} \geq (1 - 2 \epsilon)\frac{|S_1|}{|S|} \geq (1 - 2 \epsilon)(1 -  (C - 2 ) \epsilon ) \geq 1 - C\epsilon.	
\end{align}

Using Equations~\eqref{EqHMRoundCardinality}, \eqref{EqHMRoundMean} and \eqref{EqHMRoundVar}, we get that there exists a subset $S_1$ such that for all $v \in \cS^{d-1}$ and $\delta' = O( \delta + \sqrt{ r_1 d \log d /n } +r_1 r_2\sigma_k \epsilon^{1/2 - 1/k} \sqrt{d \log d/ n}+ r_1r_2\sigma_k \epsilon^{1 - \frac{1}{k}})$:
\begin{enumerate}
	\item $|S_1| \geq (1 -2 \epsilon)n \geq (1 - \epsilon') n$.
	\item $\|\mu_{S_1}\| \leq \delta' $.
	\item $v^T \overline{\Sigma}_{S_1}v \leq 1 + \frac{\delta'^2}{\epsilon'}$.
	\item For all subsets $S' \subseteq S_1: |S'| \geq (1 - \epsilon')|S_1|$,  $v^T \overline{\Sigma}_{S'}v \geq 1 - \frac{\delta'^2}{\epsilon'}$.
\end{enumerate}
We now invoke Claim~\ref{ClaimSuffStabHighMom} to conclude that $S'$ is $(\epsilon' , 7\delta')$-stable.
\end{proof}

\subsection{Upper Bound on the Second Moment Matrix}
\label{AppHighVarUpp}

\begin{lemma}
Consider the conditions in Lemma~\ref{LemStabVarianceMoment}. Then, with probability $ 1 - \tau$, $R'/n  \leq  \delta^2 / \epsilon$ ,
where $\delta = O( \sqrt{d \log d / n} + \sigma_k \epsilon^{1 - \frac{1}{k}}	  + \sigma_4\sqrt{\log(1/\tau)/n} )$.
\label{LemTruncVarEmpMoment}
\end{lemma}
\begin{proof}(Proof of Lemma~\ref{LemTruncVarEmpMoment})
We first calculate the wimpy variance required for Theorem~\ref{LemTalagrandBddArbitrarily},
\begin{align}
\sigma^2 &= \sup_{M \in M} \sum_{i=1}^n \Var( f(x_i^TMx_i)) \leq \sup_{M \in M} \sum_{i=1}^n\E f(x_i^TMx_i)^2 \\
		&\leq n \sup_{M \in M} \E (x_i^TMx_i)^{4} \leq n\sigma_{4}^{4}.
\end{align}
We use symmetrization, contraction, and matrix concentration (Theorem~\ref{LemMatrixConc}) to bound   $\E R'$ as follows:
\begin{align*}
\E R' &= \E \sup_{M \in \cM} \sum_{i=1}^n f( x_i^TMx_i) - \E f( x_i^TMx_i)  \leq 2\E \sup_{M \in \cM} \sum_{i=1}^n \epsilon_i f( x_i^TMx_i) \\
&\leq 2\E \sup_{M \in \cM} \sum_{i=1}^n \epsilon_i  x_i^TMx_i = 2\E \|\sum_{i=1}^n   \epsilon_i x_ix_i^T \| \\
&= O\left(\sqrt{  \frac{\sigma_k^2nd \log(d)}{\epsilon^{\frac{2}{k}}} } + \frac{\sigma_k^2 d\log d}{\epsilon^{\frac{2}{k}}} \right),
\end{align*}
where we use Theorem~\ref{LemMatrixConc}, with $\nu = O(\sigma_k^2nd \epsilon^{-\frac{2}{k}})$ and $L = O(\sigma_k^2 d \epsilon^{-\frac{2}{k}})$.

Note that $Q_k = O(\sigma_k \epsilon^{- \frac{1}{k}} + (1 / \epsilon) \sqrt{d/n}$. As $R'$ is bounded by $Q_k^2$, we can apply Theorem~\ref{LemTalagrandBddArbitrarily} to get that with probability at least $1 - \tau$, $R'/n$ is bounded as follows:
\begin{align*}
\frac{R'}{n} &= O\left( \sqrt{ \frac{\sigma_k^2 d\log d}{n\epsilon^{\frac{2}{k}}} } + \frac{\sigma_k^2 d\log d}{n\epsilon^{\frac{2}{k}}}  + \sigma_{4}^{2} \sqrt{ \frac{\log( \frac{1}{\tau}) } {n}} + \frac{\sigma_k^2}{\epsilon^{\frac{2}{k}}}  \frac{\log( \frac{1}{\tau}) } {n} + \frac{1}{\epsilon^2}\frac{d}{n} \frac{\log( \frac{1}{\tau}) } {n} \right)
\\&=  \frac{1}{\epsilon}O\left( \sqrt{ \frac{ d\log d}{n} }\sigma_k \epsilon^{1 - \frac{1}{k}} + \frac{ d\log d}{n}\sigma_k^2 \epsilon^{1 - \frac{2}{k}}  + \sigma_4 \epsilon \sigma_{4} \sqrt{ \frac{\log( \frac{1}{\tau}) } {n}} + \sigma_k^2 \epsilon \epsilon^{1 - \frac{2}{k}} + \frac{d}{n} \right) \tag{ \text{ Using } $\frac{\log(\frac{1}{\tau})}{n}  = O(\epsilon)$.}
\\
&= \frac{1}{ \epsilon}O\left(\left( \sqrt{\frac{d \log d}{n}} + \sigma_k \epsilon^{1 - \frac{1}{k}} + \sigma_k \epsilon^{\frac{1}{2} - \frac{1}{k}} \sqrt{\frac{d \log d}{n}} + \sigma_4 \epsilon +  \sigma_4\sqrt{\frac{\log( \frac{1}{\tau}) } {n}} \right) \right)^2\\
&= \frac{1}{ \epsilon}O\left(\left( \sqrt{\frac{d \log d}{n}} + \sigma_k \epsilon^{1 - \frac{1}{k}} +  \sigma_4\sqrt{\frac{\log( \frac{1}{\tau}) } {n}} \right) \right)^2 \tag{ \text{ Using } $\sigma_4 \epsilon \leq \sigma_k \epsilon^{1 - \frac{1}{k}}$ \text{ and } $\sigma_k \epsilon^{\frac{1}{2} - \frac{1}{k}} = O(1)$. }\\
&\leq \frac{\delta^2}{\epsilon},
\end{align*}
where we use the parameter regime stated in Lemma~\ref{LemStabVarianceMoment}.
\end{proof}

\subsection{Controlling the Mean}
\label{AppHighMean}

\begin{lemma}
Consider the setting in Lemma~\ref{LemStabMeanMoment}.
Then, with probability, $1 - \tau - \exp(- n \epsilon)$,
\begin{align*}
\frac{R'}{n} = O\Big(\sqrt{\frac{d}{n}} + \sqrt{\frac{\log(1 / \tau)} {n}} + \sigma_k \epsilon^{1 - \frac{1}{k}} \Big).
\end{align*}
\label{LemTruncMeanEmpMoment}
\end{lemma}
\begin{proof}
We first calculate the wimpy variance required for Theorem~\ref{LemTalagrandBddArbitrarily},
\begin{align*}
\sigma^2 &= \sup_{v \in \cS^{d-1}} \sum_{i=1}^n \Var( g(x_i^Tv))
		\leq \sup_{v \in \cS^{d-1}} \sum_{i=1}^n\E g(v^Tx_i)^2  \leq \sup_{v \in \cS^{d-1}} n  \E (v^Tx_i)^2  \leq n.
\end{align*}
We use symmetrization, contraction of Rademacher averages to bound $\E R'$.
 \begin{align*}
\E R' &=  \E \sup_{v \in \cS^{d-1}} \sum_{i=1}^n  g (v^T x_i) - \E g(v^Tx_i) \\
		&\leq 2 \E \sup_{v \in \cS^{d-1}} \sum_{i=1}^n \epsilon_i  g (v^T x_i)\\
		&\leq 2 \E \sup_{v \in \cS^{d-1}} \sum_{i=1}^n \epsilon_i  v^T x_i = 2 \E \|\sum_{i=1}^n \epsilon_i   x_i \| \leq 2 \sqrt{\frac{d}{n}}.
 \end{align*}
By applying Theorem~\ref{LemTalagrandBddArbitrarily}, we get that with probability at least $1 - \tau$,
\begin{align*}
\frac{R'}{n}  &= O\Big(\frac{\E R'}{n}   + \sqrt{\frac{\log(1/ \tau)} {n}} + Q_k \frac{\log(1/ \tau)}{n} \Big) \\
&= O\Big(\sqrt{\frac{d}{n}} + \sqrt{\frac{\log(1/ \tau)} {n}} + \sigma_k \epsilon^{- \frac{1}{k}} \frac{\log(\frac{1}{\tau})}{n} + \frac{1}{\epsilon} \sqrt{\frac{d}{n}}	\frac{\log(1 / \tau)}{n} \Big)\\
&= O\Big( \sqrt{\frac{d}{n}} + \sqrt{\frac{\log(1/ \tau)} {n}} + \sigma_k \epsilon^{1 - \frac{1}{k}}\Big),
 \end{align*}
where the last inequality uses the assumption that $\frac{\log(1/ \tau)} {n} = O(\epsilon)$.
\end{proof}

\end{document}